\documentclass[11pt]{article}

\usepackage{fullpage}
\usepackage{amsmath, amsfonts, amssymb, amsthm, amsxtra}
\usepackage{dsfont}
\usepackage{mathrsfs, mathtools}
\usepackage{bbm}
\usepackage{bm}
\usepackage[shortlabels]{enumitem}
\usepackage{natbib}
\usepackage[dvipsnames]{xcolor}
\usepackage[hidelinks]{hyperref}
\usepackage{comment}

\numberwithin{equation}{section}

\newtheorem{theorem}{Theorem}[section]
\newtheorem{definition}[theorem]{Definition}
\newtheorem{lemma}[theorem]{Lemma}
\newtheorem{corollary}[theorem]{Corollary}

\newtheorem{remark}[theorem]{Remark}
\newtheorem{proposition}[theorem]{Proposition}

\newtheorem{fact}[theorem]{Fact}

\newcommand{\bX}{\bm{X}}

\newcommand{\cE}{\mathcal{E}}

\newcommand{\cN}{\mathcal{N}}

\newcommand{\CC}{\mathbb{C}}

\newcommand{\EE}{\mathbb{E}}

\newcommand{\PP}{\mathbb{P}}
\newcommand{\QQ}{\mathbb{Q}}
\newcommand{\RR}{\mathbb{R}}
\def\abs#1{\left| #1 \right|}

\newcommand{\inparen}[1]{\left(#1\right)}             %\inparen{x+y}  is (x+y)
\newcommand{\inbraces}[1]{\left\{#1\right\}}           %\inbrace{x+y}  is {x+y}
\newcommand{\insquare}[1]{\left[#1\right]}             %\insquare{x+y}  is [x+y]
\newcommand{\E}{\mathbb{E}}

\renewcommand{\P}{\mathbb{P}}
\newcommand{\eps}{\varepsilon}
\newcommand{\ud}{\textnormal{d}}
\newcommand{\1}{\mathbf{1}}
\newcommand{\Bern}{\operatorname{\mathsf{Bernoulli}}}
\newcommand{\Bin}{\operatorname{\mathsf{Bin}}}

\newcommand{\Hyp}{\operatorname{\mathsf{Hypergeom}}}

\newcommand{\ER}{Erd\H{o}s--R\'enyi}

\newcommand{\Var}{\operatorname{\mathsf{Var}}}
\newcommand{\TV}{\operatorname{\mathsf{TV}}}
\newcommand{\PSCenterVertex}{\ensuremath{v_*}}

\title{Optimal Detection of Planted Stars via a \\ Random Energy Model}
\author{
Ijay Narang\thanks{School of Computer Science, Georgia Institute of Technology, inarang3@gatech.edu}
\and
Will Perkins\thanks{School of Computer Science, Georgia Institute of Technology, wperkins3@gatech.edu}
\and 
Timothy L.~H.~Wee\thanks{School of Mathematics, Georgia Institute of Technology, timothy.wee@gatech.edu}}

\date{\today}

\begin{document}

\maketitle

\begin{abstract}
We study the problem of detecting a planted star in the Erd{\H{o}}s--R{\'e}nyi random graph $G(n,m)$, formulated as a hypothesis test. We determine the scaling window for critical detection in $m$ in terms of the star size, and characterize the asymptotic total variation distance between the null and alternative hypotheses in this window. In the course of the proofs we show a condensation phase transition in the likelihood ratio that closely resembles that of the random energy model from spin glass theory. 
\end{abstract}


\tableofcontents

\section{Introduction}

Extracting information from networks is a fundamental task that spans many fields including computer science, statistics, and statistical physics. Extensive work has been done to better understand thresholds and algorithms for recovery and detection of planted structures in random graphs. Notable examples include the planted clique problem \cite{jerrum1992metropolis, AlonKrivelevichSudakov1998HiddenClique} for which many complexity-theoretic questions remain, as well as the community detection or stochastic block models \cite{decelle2011asymptotic,bhaskara2010detecting,arias2014community,mossel2015reconstruction,abbe2018community}. In recent years, planted combinatorial motifs, often with no intrinsic low rank structure, have received significant attention. These include planted Hamiltonian cycles \citep{bagaria2020hidden}, hidden nearest-neighbor graphs \citep{ding2020consistent}, planted matchings and $k$-factors \citep{moharrami2021planted,ding2023planted,gaudio2025all,wee2025cluster, addarioberry2026statisticalthresholdplantedmatchings}, and planted trees \citep{massoulie2019planting,moharrami2025planted}.

In network science, it is often of interest to identify central nodes, or hubs, which have an outsized influence on the rest of the graph. This has many applications, ranging from page ranking in search engines, to super-spreaders in epidemiological networks \cite{Kirkley2024Identifying}. Such questions have been investigated in the context of preferential or uniform attachment networks, in the form of identifying early vertices or seed graphs \citep{bubeck2017finding,RaczBubeck2017Survey, addarioberry2024optimalrootrecoveryuniform}.

This work studies such a hub detection question in the context of a planted subgraph problem: Consider a planted star of size $k$ in a graph on $n$ vertices. How many edges $m = m(n,k)$ can we throw randomly on the graph before the planted star gets buried?

A star is a tree, and thus variants of the planted star problem have been studied in \cite{massoulie2019planting} and \cite{moharrami2025planted}. The former operates in a sparse setting in the \ER~$G(n,p)$, and the latter considers spanning weighted trees. This paper instead considers planted stars of diverging size in an \ER~$G(n,m)$ setting, where the number of edges in the planted and null models coincide exactly. Aside from the clear differences between our settings, the focus of this paper is on the \emph{precise} detection behavior at the critical threshold. In this regard, much less is known, even in the planted subgraph literature in general.

\subsection{Summary of contributions}

\paragraph{Precise threshold fundamental limits of detection} We give an exact answer to the question above in the form of a critical scaling window $m = a_{n,m,k} + \gamma b_{n,m,k}$, where $\gamma$ is a constant. Outside this window, detection is shown to be either information-theoretically trivial or impossible, while within this window we characterize the total variation rate in terms of $\gamma$.
\paragraph{Likelihood ratio and REM}
We establish the asymptotics of the likelihood ratio between the planted and null distributions. Our analysis exposes a surprising connection between this likelihood ratio and the \emph{random energy model (REM)} from spin glass theory. Our approach builds upon the truncated second-moment methods used to analyze the classical REM. However, modifications are needed to address the significant dependencies between degrees of vertices that arise in the $G(n,m)$ model. These include showing that the moment generating functions of the degrees are relatively stable in the sense that conditioning on additional degree information induces only tiny perturbations; an appropriate use of higher order Bonferroni inequalities to control fluctuations of the likelihood ratio at the critical phase transition; and the use of strongly Rayleigh distributions and the geometry of polynomials to control the cumulants of the Hypergeometric distribution.

\paragraph{Simple optimal test and asymmetrical testing errors}  We show that the optimal likelihood ratio thresholding test reduces to a simpler test thresholding the maximum degree statistic. Our characterization of the likelihood ratio reveals that it is asymptotically degenerate under the null which leads to an asymmetry in the form of a vanishing Type-I and non-vanishing Type-II error. Finally, we also show the absence of a statistical-to-computational gap in this detection problem.

\subsection{Problem formulation}

Given positive integers $n$, $m$, and $k$, the \textbf{planted} distribution $\PP_1$ is the distribution over random graphs on $n$ vertices constructed as follows. First pick a hub vertex $\PSCenterVertex \in [n]$ uniformly at random and add $k$ uniformly randomly  incident edges to $\PSCenterVertex$. Then choose $m-k$ edges chosen uniformly over the remaining $N - k$ unplanted edges, where $N = \binom{n}{2}$. The \textbf{null} distribution $\PP_0$ is the \ER~random graph model $G(n,m)$: the uniform distribution over all graphs on $n$ vertices with $m$ edges.

Given a graph $G$, we consider the hypothesis test
\begin{equation}\label{eqMain_detection_problem}
H_0:\; G\sim \PP_0,
\qquad\text{vs.}\qquad
H_1:\; G\sim \PP_1.
\end{equation}

Throughout the paper, we restrict our attention to planted stars of size
\begin{align}\label{eqAsmpt_on_k}
    (\log n)^2 \ll k \ll \sqrt{n}.
\end{align}
The lower bound places us beyond the sparse regime. The upper restriction on $k$ follows from the standard deviation of degrees under both the null and alternative being upper bounded by $\sqrt{n}$.

\paragraph{Notation}
We use the standard $\omega(\cdot), o(\cdot), \Theta(\cdot), O(\cdot)$ notation. Write $a \sim b$ to mean $a/b \to 1$. We say that an event holds with high probability (whp) if it holds with probability $1 - o(1)$ as $n \to \infty$.
For any graph $G$, let $\Delta = \Delta(G)$ denote its maximum degree. The Gaussian cumulative distributive function is denoted by $\Phi(\cdot)$ and its complement is denoted $\Phi^c( \cdot)$.  

\subsection{Main results}

\begin{theorem}\label{thmMainResult_TV}
Suppose the star size $k$ satisfies \eqref{eqAsmpt_on_k}. For $\gamma \in \mathbb{R}$, set 
\begin{equation}\label{eqScalingWindowM}
    m = \frac{k^2n}{4\log n} \left (1 + \frac{\gamma}{\sqrt{\log n}} \right).
\end{equation}
Then, as $n \rightarrow \infty$,
\[
\TV(\PP_1, \PP_0)
= 
1 - \Phi \bigg(\frac{\gamma}{\sqrt{2}}\bigg) + o(1).
\]
\end{theorem}
This theorem establishes a sharp threshold for detecting planted stars and finds the precise scaling window around this sharp threshold.  

\begin{remark}
    A similar result holds in an analogous formulation in the independent edges setting. We test $H_0 : G \sim G(n,p)$ against $H_1 : G \sim \PP_1$, where a sample from the alternative is the union of a randomly planted $k$-star and $G(n,p)$. Here $k$ satisfies \eqref{eqAsmpt_on_k}, $m$ is defined by \eqref{eqScalingWindowM}, and we set $p = m/N$. Note that the spike size $k$ is dominated by the standard deviation of the edge count. Therefore it is without loss that the ambient random graph $G(n,p)$ is the same for the null and alternative. We omit the proofs because they are the same but simpler.
\end{remark}

Now consider the likelihood ratio between the planted and null distributions in \eqref{eqMain_detection_problem}:  
\begin{equation}
\label{eqLikelihood}
    \Lambda(G) = \frac{\ud \PP_1}{\ud \PP_0}(G).
\end{equation}
We recall the following characterization of total variation distance:
\begin{equation} \label{eqTVChar}
    \TV(\PP_0, \PP_1) = \mathbb P_0(\Lambda \ge 1) - \mathbb P_1(\Lambda \ge 1) \,,
\end{equation}
which is equivalent to the fact that the likelihood ratio test (choosing $H_0$ if $\Lambda(G) < 1$ and $H_1$ otherwise) achieves the minimum possible sum of Type-I and Type-II errors. 

Complementing Theorem~\ref{thmMainResult_TV}, we present a very simple test statistic for distinguishing between $H_0$ and $H_1$.  Define, for any fixed $\alpha > 1$, the null max degree threshold
\begin{equation} \label{eqThresholdDefn}
t^\ast = t^\ast(n,m,k)= \frac{2m}{n}+\sqrt{\frac{2m}{n}\Big(1-\frac{m}{N}\Big)}\,
\sqrt{2\log n-\frac{1}{\alpha}\log\log n}\, .
\end{equation}
It will be convenient to consider the critical scaling \eqref{eqScalingWindowM} in $m$ at a less granular level by defining, for $c > 0$,
\begin{equation}\label{eqCriticalWindow_m_in_c}
    m = c\frac{k^2n}{\log n}\big(1 + o(1)\big) \, 
\end{equation}

\begin{theorem}\label{thmOptimalTest}
Fix $\gamma \in \mathbb R$. Suppose the star size $k$ satisfies \eqref{eqAsmpt_on_k} and $m$ satisfies~\eqref{eqScalingWindowM}. Then thresholding the maximum degree $\Delta$ of $G$ at $t^\ast$ coincides with thresholding the likelihood ratio at $1$ with high probability under both $\PP_0$ and $\PP_1$.
That is, the probability of the symmetric difference of the events $\{\Lambda(G) \ge1 \}$ and $\{ \Delta(G) \ge t^\ast \}$ is $o(1)$ under both $\mathbb P_0$ and $\mathbb P_1$.
\end{theorem}

Thus, an immediate consequence of Theorem \ref{thmOptimalTest} is that the optimal test $\1\!\inbraces{ \Lambda(G) \geq 1 }$ that achieves the total variation rate (and hence minimizes the sum of Type-I and Type-II errors) is approximately equivalent in the critical window to the test $\1\{\Delta(G) \geq t^*\}$. Since the maximum degree is efficiently computable, this establishes that the planted star detection problem exhibits no statistical-to-computational gap, in a very strong sense: a very simple and tractable statistic asymptotically minimizes the sum of Type-I and Type-II errors throughout the detection scaling window. 

\begin{remark}
The preceding results indicate that thresholding the likelihood ratio (a degree $k \gg (\log n)^2$ polynomial \eqref{eqLR_Original}) or the maximum degree (not a polynomial) achieves asymptotically optimal testing error. These are not low degree tests, yet they are computationally efficient. In fact, asymptotically optimal low degree tests do exist. A calculation shows that summing signed star counts (suitably standardized powers of degrees) up to size $D$ succeeds for $D > C \log n$ for some universal constant $C$ depending only on $\gamma$.
\end{remark}

The key to our proof is the characterization of a phase transition of the likelihood ratio at the critical scaling. 

\begin{proposition} \label{propZphasetransition}
Suppose the planted star size $k$ satisfies \eqref{eqAsmpt_on_k}. Fix  $c>0$ and consider $m$ in \eqref{eqCriticalWindow_m_in_c}. Then we have the following cases under the null distribution $\PP_0$:
\begin{enumerate}[(i)]
    \item If $c>\tfrac14$, then $\Lambda(G) = 1+o(1)$ whp.
    \item If $c=\tfrac14$, then $\Lambda(G) = \frac{1}{2} +o(1)$ whp.
    \item If $c<\tfrac14$, then $\Lambda(G) = o(1)$ whp.
\end{enumerate}
\end{proposition}

\begin{remark}
    A consequence of Theorems \ref{thmMainResult_TV}, \ref{thmOptimalTest}, and Proposition \ref{propZphasetransition} (ii) is that throughout the entire scaling window \eqref{eqScalingWindowM} of $m$, the optimal likelihood ratio threshold test has vanishing Type-I error, since $\Lambda$ is concentrated around $\frac{1}{2}$ under $H_0$.

    This reveals an asymmetry in the optimal test, in that it is almost surely correct under the null, and it is the alternative that gives rise to the error. This is distinct from the log-Gaussian limits reminiscent of local asymptotic normality  \citep{cam1960locally} that is seen in several planted subgraph testing problems \citep{wee2025cluster} (and related asymptotic enumeration of subgraph counts \citep{Janson1994Numbers}), and in many other detection problems for instance in spiked models with low rank structure \citep{ElAlaoui20fundamental,banerjeeMa2022optimal}, and also in log-partition functions in spin glass models \citep{AbbeLiSly2021Proof}.
\end{remark}

With Theorem~\ref{thmOptimalTest} and Proposition~\ref{propZphasetransition} in hand, the only remaining ingredient needed to complete the proof of Theorem~\ref{thmMainResult_TV} is to characterize the behavior of the maximum-degree test under both $\P_0$ and $\P_1$. 
The following lemma enables this.
\begin{lemma}\label{lem:maxdeg_short}
Let $t^*$ be the threshold defined in \eqref{eqThresholdDefn}, and let $d_{\PSCenterVertex}$ be the degree of the center of the planted star under $H_1$. Then, we have the following: 
\begin{itemize}
    \item $\P_0(\Delta(G) \ge t^*)=o(1)$
    \item $\P_1(\Delta(G) \ge t^*)=\P_1(d_{\PSCenterVertex} \ge t^*)+o(1)$
\end{itemize}
\end{lemma}
Since the distribution of $d_{\PSCenterVertex}$ (centered and scaled) is asymptotically Gaussian (c.f.~Lemma \ref{lemHypTailToGaussianCDF}), we can compute the limit of $\P_1(d_{\PSCenterVertex} \ge t^*)$, and it will be seen that it is in fact the variance of $d_{\PSCenterVertex}$ that determines the width of the scaling window.

Suppose we know that $G$ is drawn from the planted model. Then our results imply that in the critical regime, we can perform a partial recovery of the planted star with an error probability that depends on our location in the scaling window.
\begin{corollary}\label{corrRecovery}
    Let $G \sim \PP_1$ and suppose $k$ satisfies \eqref{eqAsmpt_on_k} and we are in the scaling window \eqref{eqScalingWindowM} for some $\gamma \in \RR$. Then the planted hub vertex can be recovered with probability $1 - \Phi\big( \frac{\gamma}{\sqrt{2}} \big) + o(1)$. On the other hand, whp, we cannot recover even a constant fraction of the planted star leaves.
\end{corollary}

\subsection{Related work}

\paragraph{Planted subgraph detection} Aside from the  papers mentioned above that investigate detection thresholds and tests for specific planted motifs in random graphs, there have also been attempts to provide general frameworks for understanding arbitrary planted subgraphs. \cite{elimelech2025detecting} provides detection thresholds for a model which plants a subgraph of edge density differing from that of the ambient random graph. \cite{YuZadikZhang2025CountingStars} considers the effectiveness of constant-degree test statistics in the detection of arbitrary planted subgraphs in a dense \ER~$G(n,p)$. These results are not directly applicable to our setting since they are in the $G(n,p)$ setting and do not focus on the precise total variation rate in the critical regime.

\paragraph{Error curves in the critical regime} An asymptotic characterization of the likelihood ratio is a canonical problem in statistics, with a seminal result due to Wilks \cite{wilks1938large}. Such precise descriptions in the context of planted subgraph problems are somewhat rare. Notable exceptions are \cite{mossel2025weak, moitraWein2025precise,wee2025cluster}. Precise results in an abstract planted problem setting, not necessarily involving graph structure, are given in \cite{perkins2013forgetfulness}. A special case of the main result there can be interpreted as a variation of the planted star problem considered here, where the hub is known, and with only the degrees being observed. This difference leads to a different optimal statistic than the max degree test.  See also the remarks in~\cite[Section 9]{perkins2013forgetfulness} which suggested studying an unlabeled planted balls-and-bins problem and served as the original inspiration for this work.

The likelihood ratio in planted subgraph problems is intimately related to the subgraph counting problem, as evidenced by \eqref{eqLR_Original} (see also \citep{massoulie2019planting} Lemma 5). In light of this, we mention the sharp results of \cite{Janson1994Numbers} that give an asymptotic enumeration of the number of spanning trees (among other subgraphs considered). Our result Proposition \ref{propZphasetransition}  complements this by giving the asymptotics of the number of large but non-spanning stars.

\section{Connection to the random energy model}
Our proof begins with a sharp approximation of the likelihood ratio \eqref{eqLikelihood}, which we now spell out. Let $\inbraces{d_i}$ denote the degrees of the vertices of the graph $G$. Then
\begin{equation}\label{eqLR_Original}
    \Lambda(G) = \frac{\binom{N}{m}}{n\binom{n-1}{k}\binom{N-k}{m-k}} \sum_{i=1}^n \binom{d_i}{k}.
\end{equation}
In other words, it is the scaled number of $k$-stars in $G$.

Under the null distribution, the $d_i$'s are marginally $\Hyp(N,n-1,m)$ (c.f.~Definition \ref{defHypergeometric}) Their means and variances are 
\begin{equation}
\label{eqMuSigDef}
\mu := \E_0[d_i]= (n-1)p \sim \frac{2m}{n},
\quad
\sigma^2:=\Var_0(d_i)=(n-1)p(1-p)\frac{N}{N-1}\cdot\frac{N-(n-1)}{N-1}\sim \mu .
\end{equation}
Define the standardized degrees $Y_i \coloneqq (d_i - \mu)/\sigma$. The starting point of our analysis is to express the likelihood ratio $\Lambda$ in the following form. 
\begin{lemma}
\label{lemLR_expansion}
Fix $c>0$ and suppose the planted star size $k$ satisfies \eqref{eqAsmpt_on_k}, and  $m = (c+o(1)) \frac{nk^2}{\log n}$. Then, whp under both $\P_0$ and $\P_1$,
\begin{equation}
\label{asymptotic_likelihood_ratio}
\Lambda(G)
=
(1+o(1))\,
e^{-a_n^2/2}\,
\frac{1}{n}\sum_{i=1}^n e^{a_n Y_i}, \qquad  a_n = \frac{k\sigma}{\mu} \sim  \sqrt{\frac{\log n}{2c}}
\end{equation}
\end{lemma}
Thus, the randomness in the likelihood ratio essentially comes from the $\sum_{i=1}^n e^{a_n Y_i}$ term, which resembles the partition function of the random energy model (REM) from spin glass theory. The REM was originally formulated by \cite{Derrida1981REM} as a simplified model of the canonical Sherrington-Kirkpatrick spin glass \citep{SherringtonKirkpatrick1975SK}. The energies $\inbraces{H_j}$ in the REM are taken as i.i.d.~Gaussian random variables, thus abstracting away any underlying configuration geometry. More precisely, the REM partition function $Z_{\text{REM}}(\beta)$ is given by
\begin{align}\label{eqREM_definition}
    Z_{\text{REM}} = \sum_{j = 1}^{2^M}\exp\inparen{-\beta H_j}, \qquad\text{where}\qquad \inbraces{H_j}_{j=1}^{2^M} \overset{\text{iid}}{\sim} \cN(0,M),
\end{align}
where $M$ represents the number of sites in the system, and where $\beta > 0$ is a temperature parameter.  The motivation for  the distribution and scaling of the $2^M$ energy levels in the REM comes from the  Sherrington--Kirkpatrick spin glass in which spins are $\inbraces{\pm 1}$-valued and their energies are Gaussian random variables (albeit with complex correlation structure).

Despite its apparent simplicity, the REM already exhibits a phase transition as the temperature $\beta$ varies. We map the classical REM results into our notation as follows.

\begin{lemma}[\cite{BovierKurkovaLowe2002Fluctuations}, \cite{Bovier2006Disordered} Theorem 9.2.1]\label{lemREM-fluctuations}
Setting $M = \log n$ and $\beta = 1/\sqrt{2c}$ in \eqref{eqREM_definition}, the partition function of the REM \eqref{eqREM_definition} has the following phase transitions.
\begin{enumerate}[(i)]
    \item If $c>\tfrac14$, then $Z_{\text{REM}} = \EE[Z_{\text{REM}}] + o\!\inparen{\EE[Z_{\text{REM}}]}$ whp.
    \item If $c=\tfrac14$, then $Z_{\text{REM}} = \frac{1}{2}\EE[Z_{\text{REM}}] + o\!\inparen{\EE[Z_{\text{REM}}]}$ whp.
    \item If $c<\tfrac14$, then $Z_{\text{REM}} = o\inparen{\EE[Z_{\text{REM}}]}$ whp.
\end{enumerate}
\end{lemma}
Note that \cite{Bovier2006Disordered} Theorem 9.2.1 provides much more information by determining the exact fluctuation scales and asymptotic characterizations $Z_{\text{REM}}$ in different regimes of $c$. In particular, it is shown that below the critical threshold $c < \tfrac{1}{4}$, the random variable $Z_{\text{REM}}/\EE[Z_{\text{REM}}]$ has fluctuations of size $\exp(-\omega(\log n))$, while in the critical phase $c = \tfrac{1}{4}$ the random variable $Z_{\text{REM}}/\EE[Z_{\text{REM}}] - \tfrac{1}{2}$ has fluctuations of size $\Theta((\log n)^{-1/2})$, and above the critical threshold $c > \tfrac{1}{4}$, the fluctuations of $Z_{\text{REM}}/\EE[Z_{\text{REM}}]  $ are of size $o\!\inparen{ (\log n)^{-1/2}  }$

Our results show an interesting connection between spin glass phase transitions and detection thresholds for statistical estimation problems. More specifically, the high-temperature REM phase (Lemma \ref{lemREM-fluctuations} (i)) displays a law of large numbers phenomenon for $Z_{\text{REM}}$ (i.e.~the annealed = quenched phase in statistical physics). This coincides with the impossibility of detection phase in Proposition \ref{propZphasetransition} (i). On the other hand, the \emph{condensation} (or low temperature) REM phase (Lemma \ref{lemREM-fluctuations} (iii)) has $Z_{\text{REM}}$ being carried by the extremal energy values. This coincides with the trivial detection phase in Proposition \ref{propZphasetransition} (iii) where, in analogy, the likelihood ratio test is easily able to detect anomalies in the maximum degree.  One can compare this to the behavior of the sparse, disassortative stochastic block model in which the information-theoretic threshold for detection coincides with a condensation phase transition for the corresponding anti-ferromagnetic Ising or Potts model on a random graph~\cite{mossel2015reconstruction,coja2018information}.

We remark that the proof of Proposition \ref{propZphasetransition} does not involve a direct reduction to the REM setting. Instead we have to contend with nontrivial correlations among the centered and standardized degrees $Y_i$ in $G(n,m)$. Our proofs use a truncated second-moment method, similar to the classical REM setting. However, we offer two technical innovations to control the effect of the dependencies. First, we establish a collection of ``local insensitivity" lemmas showing that conditioning on the degree information of one vertex
only perturbs the remaining degree distribution in a negligible way in our parameter range.
Secondly, we bound the relevant truncated moments using Bonferroni inequalities:, interestingly, away from the critical threshold $c \neq \tfrac{1}{4}$, a first-order Bonferroni expansion suffices, while at the critical threshold $c = \frac{1}{4}$ we require a second-order expansion to capture the correct cancellation and obtain concentration.

The classical REM has been generalized substantially, for instance to the  generalized or continuous REMs \citep{DerridaGardner1986GREM,OlivieriPicco1984ThermoREM,CapocacciaCassandroPicco1987ThermoGREM,Ruelle1987Reformulation,BolthausenSznitman1998Cascades,BovierKurkova2004GREM1}. These models consider Gaussian energies \eqref{eqREM_definition} that have a hierarchical correlation structure, in an effort to study replica-symmetry breaking. We leave the exploration of further connections between planted subgraph problems and spin glass models such as the GREM's to  future work.

\section{Simplifying the likelihood ratio}

In this section we prove Lemma \ref{lemLR_expansion}.

\begin{definition}\label{defHypergeometric}
    A random variable $X$ has the \textbf{Hypergeometric} distribution $\Hyp(N,K,n)$ if $X$ is the number of objects with a desired property in a sample without replacement of size $n$ from a population of size $N$, where $K$ objects in the population have that desired property. 
\end{definition}
We collect further properties of hypergeomeric random variables in Appendix \ref{secApproxHypergeom}.

\begin{proof}[Proof of Lemma \ref{lemLR_expansion}]
Given a graph $G$, the probability mass functions under the null and alternative are
\[
    \ud \PP_0(G) = \frac{1}{\binom{N}{m}}, \qquad\text{and}\qquad \ud \P_1(G)
    = \frac{ 1}
       {n\binom{n-1}{k}\binom{N-k}{m-k}} \sum_{i=1}^n \binom{d_i}{k}.
\]
Therefore, the likelihood ratio is, with $\mu$ defined in \eqref{eqMuSigDef},
\begin{align}
\label{eq:LR_exact}
\Lambda(G)
    &= \frac{\binom{N}{m}}{n\binom{n-1}{k}\binom{N-k}{m-k}}
      \sum_{i=1}^n \binom{d_i}{k}  = \frac{(N)_k}{(m)_k}\,\frac{1}{n(n-1)_k}
      \sum_{i=1}^n (d_i)_k  = \frac{(N)_k}{(m)_k}\,\frac{\mu^k}{n(n-1)_k}
      \sum_{i=1}^n \frac{(d_i)_k}{\mu^k}.
\end{align}
where $(a)_b$ denotes the falling factorial. Note that
\begin{align} \label{eq:pre_taylor_LR}
\frac{1}{n}\sum_{i=1}^n \frac{(d_i)_k}{\mu^k}
&=
\frac{1}{n}\sum_{i=1}^n \prod_{j=0}^{k-1}\frac{d_i-j}{\mu}
=
\frac{1}{n}\sum_{i=1}^n \prod_{j=0}^{k-1}\Bigl(1+\frac{d_i-\mu-j}{\mu}\Bigr) \nonumber \\[0.6em]
&=
\frac{1}{n}\sum_{i=1}^n
\exp \Biggl\{\sum_{j=0}^{k-1}\log\Bigl(1+u_{i,j}\Bigr)\Biggr\},
\qquad \text{where }\,
u_{i,j}:=\frac{d_i-\mu-j}{\mu}.
\end{align}
By a union bound, Lemma \ref{lemHypTailToGaussianCDF}, and Mills ratio for the Gaussian tails, we have that whp under $\PP_0$,
\begin{align} \label{high_prob_event}
\max_{1\le i\le n}|d_i-\mu|
\le
10\,\sigma\sqrt{\log n}.
\end{align}
This same event also holds whp under $\PP_1$: by a union bound, it remains to argue for the degree of the planted center $\PSCenterVertex$. We have
\[
\PP[d_{\PSCenterVertex} - \mu > 10 \sigma \sqrt{\log n}] = \PP\insquare{\frac{d_{v_*} - \mu - k}{\sigma} > 10\sqrt{\log n} - \frac{k}{\sigma}} = (1+o(1))\Phi^c\inparen{10\sqrt{\log n} - \frac{k}{\sigma}},
\]
and the right-hand side is $O\!\inparen{\frac{1}{\sqrt{\log n}}}$ by Mill's ratio. The lower tail is argued similarly.

On the event \eqref{high_prob_event},
\[
|u_{i,j}|
\le
\frac{|d_i-\mu|+j}{\mu}
\le
\frac{10\sigma\sqrt{\log n}+k}{\mu} = o(1).
\]
as $k \gg (\log n)^2$. Thus, we can Taylor expand $\sum_{j=0}^{k-1}\log(1+u_{i,j})$ as follows:
\begin{align}
\sum_{j=0}^{k-1}\log(1+u_{i,j})
=
\underbrace{\sum_{j=0}^{k-1}u_{i,j}}_L
-\frac12\underbrace{\sum_{j=0}^{k-1}u_{i,j}^2}_Q
+ o(1)
\end{align}
where the $o(1)$ is because the remainder satisfies
\[
\sum_{j=0}^{k-1}O(u_{i,j}^3)
=
O\!\left(k\max_j|u_{i,j}|^3\right)
=
O\!\left(\frac{(\log n)^3}{k^2}\right) = o(1)
\]
again using $k\gg(\log n)^{2}$. Furthermore, for the quadratic contribution $Q$, we have
\begin{align} \label{Q}
Q
&:= \sum_{j=0}^{k-1} u_{i,j}^2 = \frac{k(d_i-\mu)^2}{\mu^2}
-\frac{k(k-1)(d_i-\mu)}{\mu^2}
+\frac{(k-1)k(2k-1)}{6\mu^2} \nonumber\\
&= \frac{k\sigma^2}{\mu^2}\,Y_i^2
-\frac{k(k-1)\sigma }{\mu^2} Y_i
+\frac{(k-1)k(2k-1)}{6\mu^2} \nonumber \\
&= o(1)
\end{align}
as on the event described in \eqref{high_prob_event}, the leading contribution here is $
\frac{k\sigma^2}{\mu^2}\,Y_i^2$
which is of order $(\log n)^2/k$ uniformly in $i$. Because $(\log n)^2 \ll k\ll \sqrt{n}$, the entire line is $o(1)$ uniformly. Now, for the linear term $L$:
\begin{align} \label{L}
L &= \sum_{j=0}^{k-1}u_{i,j} =
\frac{1}{\mu}\sum_{j=0}^{k-1}\bigl(d_i-\mu-j\bigr)=
\frac{k(d_i-\mu)}{\mu}-\frac{k(k-1)}{2\mu} = a Y_i - \frac{a^2}{2} + o(1)
\end{align}
Thus, by \eqref{Q} and \eqref{L}, we obtain
\[
\sum_{j=0}^{k-1}\log(1+u_{i,j})
=
a_n Y_i-\frac{a_n^2}{2}+o(1).
\]
implying
\[
\Lambda(G) = \frac{1}{n}\sum_{i=1}^n \frac{(d_i)_k}{\mu^k}
=
(1+o(1))\,
\frac{1}{n}\sum_{i=1}^n
\exp\!\left(a_n Y_i-\frac{a_n^2}{2}\right).
\]
\end{proof}

\section{Condensation phase transitions for the likelihood ratio}

In this section we prove Proposition \ref{propZphasetransition}. Throughout, we consider only the null model $\PP_0$ so that all expectations and probabilities are taken with respect to $\PP_0$, unless otherwise stated.

\begin{proof}[Proof of Proposition \ref{propZphasetransition}]
By Lemma \ref{lemLR_expansion} we have the following representation of the likelihood ratio: 
\begin{equation*}
\Lambda(G)
=
(1+o(1))\,
e^{-a_n^2/2}\,
\frac{1}{n}\sum_{i=1}^n e^{a_n Y_i}, \qquad a_n = \frac{k\sigma}{\mu} = (1 + o(1))  \sqrt{\frac{\log n}{2c}}
\end{equation*}
We will prove the result on a regime-by-regime basis. Our approach is an extension of the truncated second moment method used for the classical REM (c.f.~\cite{Bovier2006Disordered} Theorem 9.2.1). Let $Z_n = \sum_{i=1}^{n} e^{a_n Y_i}$. To reduce clutter, we often write $Z := Z_n$ and $a := a_n$.

The following shorthands are used throughout this proof:
\[
M_1:=\E\big[e^{aY_1}\big],
\quad
M_2:=\E\big[e^{2aY_1}\big],
\quad
M_{12}:=\E\big[e^{a(Y_1 + Y_2)}\big].
\]
Note by linearity of expectation that
\[
\E[Z]=nM_1,
\qquad
\E[Z^2]=nM_2+n(n-1)M_{12}.
\]
We first compute the asymptotics of $M_1$, $M_2$, and $M_{12}$. By Lemma \ref{lemHyp_Cum_CGF_MGF_Bound}, we have
\begin{equation}\label{eq:M1M2}
M_1=\exp\!\left(\frac{a^2}{2}\right)(1+o(1))=n^{\frac{1}{4c}}(1+o(1)),
\quad
M_2=\exp(2a^2)(1+o(1))=n^{\frac{1}{c}}(1+o(1)).
\end{equation}
It remains to compute $M_{12}$. Define $A_{12} = \1\!\inbraces{\inbraces{1,2} \in G}$. Note that $\PP[A_{12} = 1] = p = m/N$. Define $\cE_{12}$ to be the set of possible incident edges to vertices $1$ and $2$, excluding $\inbraces{1,2}$ i.e.
\[
\cE_{12} = {\{1,j\}:j\ge3\}\cup\{\{2,j\}:j\ge3\} }.
\]
Note that $\abs{\cE} = 2n - 4$. Conditionally on $A_{12} = b \in \{ 0,1\}$, define 
\[
H_b:=  G \cap \cE_{12}.
\]
Then
\[
H_b\sim\Hyp(N-1,\;2n-4,\;m-b),
\qquad
d_1+d_2=H_b+2b,
\]
and thus, writing $\lambda = a/\sigma$ for short,
\begin{align*}
M_{12}
&=
e^{-2\lambda\mu}
\Big[
(1-p)\,\E[e^{\lambda H_0}]
+
p\,e^{2\lambda}\,\E[e^{\lambda H_1}]
\Big] \\
&= 
e^{-2\lambda\mu}
\Big[
(1-p)\,\exp(2\lambda\mu+a^2+o(1))
+
p\,e^{2\lambda}\,\exp(2\lambda\mu+a^2+o(1))
\Big] \\
&=
e^{-2\lambda\mu}\,\exp(2\lambda\mu+a^2+o(1))\Big[(1-p)+p e^{2\lambda}\Big] \\
&=
\exp\!\left(a^2+o(1)\right)\Big[1+p\big(e^{2\lambda}-1\big)\Big] \\
&=
\exp\!\left(a^2+o(1)\right)\,(1+o(1))
=
e^{a^2}(1+o(1))
=
n^{\frac{1}{2c}}(1+o(1)).
\end{align*}
where the second equality follows from Lemma \ref{lemHyp_Cum_CGF_MGF_Bound}.
\paragraph{Case $c>\tfrac12$}
In this regime, the (vanilla) second moment method suffices:
\[
\frac{\Var(Z)}{(\E Z)^2}
=
\frac{\E[Z^2]}{(\E Z)^2}-1
\le
\frac{M_2}{nM_1^2}+o(1)
=
n^{-1+\frac{1}{2c}}(1+o(1)).
\]
If $c>\frac12$ then $-1+\frac{1}{2c}<0$, so the ratio is $o(1)$.
By Chebyshev's inequality,
\[
Z=\E Z +o(\E Z)
\qquad\text{whp}.
\]
\paragraph{Case $c\leq\tfrac12$} We will now proceed to the truncated second moment method. Define the truncation threshold and event:
\[
t_n:=\sqrt{2\log n}(1 + o(1)),
\qquad
B:=\Big\{\max_{1\le i\le n} Y_i\ge t_n\Big\}.
\]
Note that $B^c = \Big\{\max_i Y_i<t_n\Big\}$ is a high probability event: by union bound, Lemma \ref{lemHypTailToGaussianCDF}, and Mill's ratio
\begin{align}\label{eqHyp_tn_MillsRatioBound}
\PP(B) &\le n\,\PP(Y_1\ge t_n) \leq \frac{(1+o(1))}{\sqrt{\log n}}.
\end{align}
To complete the proof, we will study the behavior of $Z \mid B^c$. The conditional first moment can be upper bounded by
\begin{align}
\E[Z\mathbf 1_B]
&=
\E\!\left[\sum_{i=1}^n e^{aY_i}\mathbf 1_B\right]
\le
\E\!\left[\sum_{i=1}^n e^{aY_i}\sum_{j=1}^n \mathbf 1_{\{Y_j\ge t_n\}}\right] \nonumber \\
&=
\underbrace{n\,\E\!\left[e^{aY_1}\mathbf 1_{\{Y_1\ge t_n\}}\right]}_{S_{\text{diag}}}
+
\underbrace{n(n-1)\,\E\!\left[e^{aY_1}\mathbf 1_{\{Y_2\ge t_n\}}\right]}_{S_{\text{ off}}}.
\end{align}
Let $\QQ$ denote the distribution of $Y_1$ under $\PP_0$. 
For each $a = \Theta(t_n)$, define the tilted measure $\QQ^{(a)}$ which has density $\frac{\ud \QQ^{(a)}}{\ud \QQ}(y) = e^{ay}/\EE[e^{aY_1}]$ with respect to $\QQ$. We have
\[
S_{\text{diag}} = n\E\!\left[e^{aY_1}\mathbf 1_{\{Y_1\ge t_n\}}\right] = nM_1\cdot \QQ^{(a)}(Y_1\ge t_n) = (1+o(1)) n M_1 \cdot \Phi^c(t_n - a_n),
\]
where the last equality follows from Lemma \ref{lemHypTailToGaussianCDF}.  For $S_{\text{off}}$, we have (by \ref{lem:mgf-stab-d2}) and Lemma \ref{lemHypTailToGaussianCDF} again
\begin{align} 
S_{\text{off}} &= n(n-1)\E\!\left[e^{aY_1}\mathbf 1_{\{Y_2\ge t_n\}}\right] \nonumber = n(n-1)
M_1(1+o(1))\PP(Y_2\ge t_n) \nonumber = (1 +o(1))n^2M_1 \Phi^c(t_n).
\end{align}
By \eqref{eqHyp_tn_MillsRatioBound} and $\E Z = nM_1 = n^{1+\frac{1}{2c}}(1+o(1))$ we see that $S_{\text{off}} = O\!\inparen{\frac{1}{\sqrt{\log n}}}$. Note that 
\[
\E\!\left[Z\1_B\right]
\;\ge\;
\sum_{i=1}^n \E\!\left[e^{aY_i}\,\mathbf 1_{\{Y_i\ge t_n\}}\right]
\;=\;
n\,\E\!\left[e^{aY_1}\,\mathbf 1_{\{Y_1\ge t_n\}}\right]
\;=\;
S_{\rm diag}.
\]
It follows that $S_{\text{diag}} \leq \EE[Z \1_B] \leq S_{\text{diag}} + O\!\inparen{\frac{1}{\sqrt{\log n}}}$. Applying the fact that $\E [Z | B^c] = (1 + o(1))(\E[Z] - \E[Z\1_B])$ yields 
\begin{align}
\E[Z | B^c] =
\begin{cases}
\E[Z] + o(\E[Z]), & \text{if } \tfrac14 < c \leq \tfrac12,\\
\tfrac12 \E[Z] + o(\E[Z]), & \text{if } c = \tfrac14,\\
o(\E[Z]), & \text{if } 0 < c < \tfrac14,
\end{cases}
\end{align}
with explanations as follows.
\begin{itemize}
\item If $\tfrac14 < c \leq \tfrac12$, then $t_n-a_n\to +\infty$. Hence $\Phi^c(t_n-a_n)=o(1)$ and $\E[Z\mathbf 1_B]=o(\E Z)$. It follows that $\E[Z\mid B^c]=\E Z+o(\E Z)$).
\item If $c=\tfrac14$, then $t_n-a_n\to 0$. Hence $\Phi^c(t_n-a_n)\to \tfrac12$ and $\E[Z\mathbf 1_B]=\tfrac12\,\E Z+o(\E Z)$. It follows that $\E[Z\mid B^c]=\tfrac12\,\E Z+o(\E Z)$).
\item If $0<c<\tfrac14$, then $t_n-a_n\to -\infty$. Hence $\Phi^c(t_n-a_n)\to 1$ and $\E[Z\mathbf 1_B]=\E Z+o(\E Z)$. It follows that $\E[Z\mid B^c]=o(\E Z)$).
\end{itemize}
Therefore, all that remains is to show concentration for $Z \mid B^c$. We will now show this on a regime-by-regime basis.

\paragraph{Concentration in the regime $\tfrac14 < c \leq \tfrac12$.}
Note that
\begin{align}
\E[Z_n^2\mathbf 1_B]
&\le
\sum_{i,j,k=1}^n \E\!\left[e^{aY_i+aY_j}\mathbf 1_{\{Y_k\ge t_n\}}\right] \nonumber\\
&=
n\,T_{111}+n(n-1)\,T_{112}+2n(n-1)\,T_{121}+n(n-1)(n-2)\,T_{123}, \label{eq:EI2B-decomp}
\end{align}
where
\[
T_{111}:=\E\!\left[e^{2aY_1}\mathbf 1_{\{Y_1\ge t_n\}}\right],\quad
T_{112}:=\E\!\left[e^{2aY_1}\mathbf 1_{\{Y_2\ge t_n\}}\right],
\]
\[
T_{121}:=\E\!\left[e^{aY_1+aY_2}\mathbf 1_{\{Y_1\ge t_n\}}\right],\quad
T_{123}:=\E\!\left[e^{aY_1+aY_2}\mathbf 1_{\{Y_3\ge t_n\}}\right].
\]
We will now compute the above four terms.
\begin{enumerate}[(i)]
\item $T_{111}$. By exponential tilting and Lemma \ref{lemHypTailToGaussianCDF}
\begin{align}
T_{111} &= M_2\cdot \QQ^{(2a)}(Y_1\ge t_n) = (1 + o(1))M_2 \Phi^c(t_n - 2a_n).
\end{align}
\item $T_{112}$. By Lemmas \ref{lem:mgf-stab-d2} and \ref{lemHypTailToGaussianCDF} we have that
\begin{align} \label{eq:T112-asym}
T_{112}&=M_2(1+o(1))\,\PP(Y_2\ge t_n) = (1 + o(1))M_2 \Phi^c(t_n)
\end{align}
\item $T_{121}$. Conditioning on $d_1=k$, by Lemmas \ref{lem:mgf-stab-d2} and \ref{lemHypTailToGaussianCDF}
\begin{align}
T_{121}
&= \E\!\left[e^{aY_1+aY_2}\mathbf 1_{\{Y_1\ge t_n\}}\right] =
\sum_{k\ge \mu+\sigma t_n}\PP(d_1=k)\,
e^{\frac{a(k-\mu)}{\sigma}}\,
\E\!\left[e^{a Y_2}\mid d_1=k\right] \nonumber\\
&= (1+o(1)) M_1\,\E\!\left[e^{aY_1}\mathbf 1_{\{Y_1\ge t_n\}}\right] = (1+o(1)) M_1^2\,\QQ^{(a)}(Y_1\ge t_n) \nonumber\\
&= (1 + o(1)) M_1^2  \Phi^c(t_n - a_n).
\end{align}
\item $T_{123}$. By similar conditioning and Lemmas \ref{lem:mgf-stab-d3} and \ref{lemHypTailToGaussianCDF},
\begin{align}\label{eqT123}
T_{123}&= 
\E\!\left[e^{aY_1+aY_2}\1\{Y_3\ge t_n\}\right] 
\nonumber \\
&=\sum_{k\ge \mu +\sigma t_n}\Pr(d_3=k)\,
\E\!\left[\exp\!\bigg(\frac{a}{\sigma}(d_1+d_2-2\mu)\bigg)\,\Big|\, d_3=k\right] \nonumber \\
&=(1+o(1))\,M_{12}\PP\{Y_3 \geq t_n\} \nonumber \\
&= (1+o(1))\,M_{12}\Phi^c(t_n)
\end{align}
\end{enumerate}
For $\tfrac14<c\le \tfrac12$, we have $t_n-2a_n\to-\infty$, and $t_n - a_n \to \infty$, which implies
\begin{align}
nT_{111}
&=(1+o(1))\,nM_2 = (1+o(1))\,nM_2, \notag\\
n(n-1)T_{112}
&= (1+o(1))\,n^2M_2\cdot \frac{n^{-1}}{\sqrt{\log n}} =o(nM_2), \notag\\[0.5em]
2n(n-1)T_{121}
&= o(n^2M_1^2) = o(n^2M_{12}), \notag\\[0.5em]
n(n-1)(n-2)T_{123}
&= (1+o(1))\,n^3M_{12}\cdot \frac{n^{-1}}{\sqrt{\log n}}  = o(n^2M_{12}). 
\end{align}
Because we have the lower bound $\E[Z_n\1_B] \geq nT_{111} = (1 + o(1))nM_2$, we obtain
\[
\E[Z_n^2 \mid B^c]
= 
(1 + o(1))\Big(\E[Z_n^2]-\E[Z_n^2\mathbf 1_B]\Big)
=
(1 + o(1))n^2M_{12}
\]
which gives
\begin{align}
\frac{\Var(Z_n \mid B^c)}{\E[Z_n \mid B^c]^2}
&=\frac{\E[Z_n^2 \mid B^c]}{\E[Z_n \mid B^c]^2}-1 =\left(1+o(1)\right)\frac{n^{1/(2c)}}{n^{1/(2c)}}-1 =o(1) \notag
\end{align}
implying concentration by Chebyshev's inequality.

\paragraph{Concentration in the regime $c=\tfrac14$.}

In the first regime, we were able to show concentration via a ``first-order inclusion-exclusion" expansion. In this critical regime, we must expand to the second order. This is because
\[
\E[Z_n\mid B^c]=\tfrac12\,\E[Z_n]+o(\E[Z_n])
\]
so in order to get the correct cancellation for Chebyshev's, we must recover the constant on the second conditional moment. Since $Z_n=\sum_{i=1}^n e^{aY_i}$, we have
\[
Z_n^2=\sum_{i=1}^n e^{2aY_i}+\sum_{i\ne j}e^{aY_i+aY_j},
\]
and thus
\begin{equation}\label{eq:EI2Bc-split}
\E[Z_n^2\mathbf 1_{B^c}]
=
\underbrace{\sum_{i=1}^n \E\!\left[e^{2aY_i}\mathbf 1_{B^c}\right]}_{=:D}
+
\underbrace{\sum_{i\ne j}\E\!\left[e^{aY_i+aY_j}\mathbf 1_{B^c}\right]}_{=:O}.
\end{equation}
By linearity of expectation, these diagonal and off-diagonal terms are
\begin{equation}\label{eq:Dbc-Obc}
D=n\,\E\!\left[e^{2aY_1}\mathbf 1_{B^c}\right],
\qquad
O=n(n-1)\,\E\!\left[e^{aY_1+aY_2}\mathbf 1_{B^c}\right].
\end{equation}
\begin{enumerate}[(i)]
    \item We give an upper bound for term $D$. Since $B^c\subseteq\{Y_1<t_n\}$, by Lemma \ref{lemHypTailToGaussianCDF}
\begin{align}
\E\!\left[e^{2aY_1}\mathbf 1_{B^c}\right]
&\le \E\!\left[e^{2aY_1}\mathbf 1_{\{Y_1<t_n\}}\right]
= M_2 \QQ^{(2a)}(Y_1<t_n) \nonumber\\
&=(1+o(1))\,M_2\,\Phi(t_n-2a_n) \nonumber\\
&=(1+o(1))\,M_2\,\Phi(-(1+o(1))\sqrt{\log n}) \nonumber \\
&= (1 + o(1))M_2 \Phi^c((1+o(1))\sqrt{\log n})
\end{align}
By Mill's ratio for the Gaussian tail, we get
\begin{equation}\label{eq:D-negligible}
D
\le
(1+o(1))\,nM_2\cdot \frac{n^{-1}}{\sqrt{\log n}}
=
\frac{1+o(1)}{\sqrt{\log n}}\,M_2
=
o(n^2M_{12}),
\end{equation}
with the last equality from $M_2=n^{4}(1+o(1))$ and $n^2M_{12}=n^4(1+o(1))$ when $c=\tfrac14$.

\item We compute term $O$. Write
\[
B = \Big(\{Y_1\ge t_n\}\cup\{Y_2\ge t_n\}\Big)\ \cup\ \Big(\bigcup_{k=3}^n\{Y_k\ge t_n\}\Big),
\]
and denote $B_{12}:=\{Y_1\ge t_n\}\cup\{Y_2\ge t_n\}$ and $C_{12}:=\bigcup_{k=3}^n\{Y_k\ge t_n\}$.
Then
\[
\mathbf 1_B=\mathbf 1_{B_{12}}+\mathbf 1_{C_{12}}-\mathbf 1_{B_{12}\cap C_{12}}
\quad\Longrightarrow\quad
\mathbf 1_B=\mathbf 1_{B_{12}}+O(\mathbf 1_{C_{12}}).
\]
and therefore, 
\begin{equation}\label{eq:B-reduction}
\E\!\left[e^{aY_1+aY_2}\mathbf 1_B\right]
=
\E\!\left[e^{aY_1+aY_2}\mathbf 1_{B_{12}}\right]
+
O\!\left(\E\!\left[e^{aY_1+aY_2}\mathbf 1_{C_{12}}\right]\right).
\end{equation}
Applying a union bound and noting that the resulting term matches $T_{123}$ (in \eqref{eqT123}), and using Lemma \ref{lemHypTailToGaussianCDF} gives
\begin{align}
\E\!\left[e^{aY_1+aY_2}\mathbf 1_{C_{12}}\right]
&\le \sum_{k=3}^n \E\!\left[e^{aY_1+aY_2}\mathbf 1_{\{Y_k\ge t_n\}}\right]
=(n-2)\,T_{123} \nonumber\\
&=(1+o(1))\,(n-2)\,M_{12}\PP(Y_3\ge t_n) \nonumber \\
&= (1 + o(1))nM_{12} \Phi^c(t_n) 
\nonumber \\
&= o(M_{12}). \label{eq:C12-small}
\end{align}
Combining \eqref{eq:B-reduction} and \eqref{eq:C12-small} gives
\begin{equation}\label{eq:B-reduction-final}
\E\!\left[e^{aY_1+aY_2}\mathbf 1_B\right]
=
\E\!\left[e^{aY_1+aY_2}\mathbf 1_{B_{12}}\right]
+o(M_{12}).
\end{equation}

\noindent By inclusion-exclusion, we have $\mathbf 1_{B_{12}}=\mathbf 1_{\{Y_1\ge t_n\}}+\mathbf 1_{\{Y_2\ge t_n\}}-\mathbf 1_{\{Y_1\ge t_n,Y_2\ge t_n\}}$, which means
\begin{equation}\label{eq:bonferroni}
\E\!\left[e^{aY_1+aY_2}\mathbf 1_{B_{12}}\right]
=
2\,\E\!\left[e^{aY_1+aY_2}\mathbf 1_{\{Y_1\ge t_n\}}\right]
-
\E\!\left[e^{aY_1+aY_2}\mathbf 1_{\{Y_1\ge t_n,\;Y_2\ge t_n\}}\right]
\end{equation}
by symmetry. Observing that the first term is $2T_{121}$, we compute
\begin{align}\label{eq:T121-half}
T_{121}
&=
(1+o(1)) M_1^2\,\PP(Y_1\ge 0) = (1 + o(1))M_1^2 \Phi^c(0)
=(1+o(1))
\tfrac12\,M_1^2
\end{align}
where the second equality follows from the fact that $t_n-a_n \rightarrow0$ when $c=\tfrac14$. For the intersection term, define
\[
J_{12}:=\E\!\left[e^{aY_1+aY_2}\mathbf 1_{\{Y_1\ge t_n,\;Y_2\ge t_n\}}\right].
\]
Conditioning on the degree $d_1 = k$, we get, writing $\lambda = \frac{a}{\sigma}$ for short,
\begin{align}
J_{12}
&=\E\!\left[e^{aY_1+aY_2}\mathbf 1_{\{Y_1\ge t_n,\;Y_2\ge t_n\}}\right] \nonumber\\
&=\sum_{k\ge \mu+\sigma t_n}
\E\!\left[e^{aY_1+aY_2}\mathbf 1_{\{Y_1\ge t_n,\;Y_2\ge t_n\}}\;\middle|\; d_1=k\right]\PP(d_1=k) \nonumber\\
&=\sum_{k\ge \mu+\sigma t_n}
e^{\lambda(k-\mu)}\,
\E\!\left[e^{\lambda(d_2-\mu)}\mathbf 1_{\{d_2\ge \mu+\sigma t_n\}}\;\middle|\; d_1=k\right]\PP(d_1=k) \nonumber\\
&=(1 + o(1)) \E\!\left[e^{\lambda(d_2-\mu)}\mathbf 1_{\{d_2\ge \mu+\sigma t_n\}}\right] \sum_{k\ge \mu+\sigma t_n}
e^{\lambda(k-\mu)}\PP(d_1=k) \nonumber\\
&= (1 + o(1))\E\!\left[e^{aY_1}\mathbf 1_{\{Y_1 \geq t_n\}}\right]^2 \nonumber \nonumber \\
&= (\tfrac14 + o(1))M_1^2 \nonumber \\
&= (\tfrac14 + o(1))M_{12}
\label{eq:J12-chain}
\end{align}
Putting everything together and substituting gives
\begin{align}
\E\!\left[e^{aY_1+aY_2}\mathbf 1_{B^c}\right]
&=
\E[e^{aY_1+aY_2}]
-
\E\!\left[e^{aY_1+aY_2}\mathbf 1_B\right] \nonumber\\
&=
M_{12}-\left(\tfrac34+o(1)\right)M_{12}
=
\left(\tfrac14+o(1)\right)M_{12}. \label{eq:pair-on-Bc}
\end{align}
Therefore,
\begin{equation}\label{eq:O-leading}
O
=
n(n-1)\,\E\!\left[e^{aY_1+aY_2}\mathbf 1_{B^c}\right]
=
\left(\tfrac14+o(1)\right)n^2M_{12}.
\end{equation}
\end{enumerate}
Combining \eqref{eq:EI2Bc-split}, \eqref{eq:D-negligible}, and \eqref{eq:O-leading} gives
\begin{equation}\label{eq:EI2Bc-final}
\E[Z_n^2\mathbf 1_{B^c}]
=
\left(\tfrac14+o(1)\right)n^2M_{12}.
\end{equation}
Since $\PP(B^c)=1-o(1)$, we conclude
\[
\E[Z_n^2\mid B^c]
=
\left(\tfrac14+o(1)\right)n^2M_{12}.
\]
and the result follows by Chebyshev's inequality.
\paragraph{Concentration in the regime $0<c<\tfrac14$.}

From the first-moment computation already obtained above,
\[
\E[Z_n\mid B^c]=o(\E[Z_n]).
\]
Since $Z_n\ge 0$, Markov's inequality on $B^c$ gives, for any fixed $\varepsilon>0$,
\[
\PP\!\left(\left.Z_n>\varepsilon\,\E[Z_n]\ \right|\ B^c\right)
\le
\frac{\E[Z_n\mid B^c]}{\varepsilon\,\E[Z_n]}
=
o(1).
\]
Therefore
\[
\PP\!\left(Z_n>\varepsilon\,\E[Z_n]\right)
\le
\PP(B)+\PP\!\left(\left.Z_n>\varepsilon\,\E[Z_n]\ \right|\ B^c\right)
=
o(1),
\]
since $\PP(B)=o(1)$.
As $\varepsilon>0$ is arbitrary, this is exactly
\[
I_n=o(\E[I_n])\qquad\text{w.h.p.}
\]
as desired.

\end{proof}

\section{Total variation bounds}

In this section we furnish the proofs of the main results. Throughout, we assume that the condition \eqref{eqAsmpt_on_k} on $k$ is in force, and that we are operating in the critical scaling window for $m$ in \eqref{eqScalingWindowM}.

\subsection{A lower bound}\label{secLowerBound}

In this section we establish the easier direction of Theorem \ref{thmMainResult_TV} by exhibiting a test that achieves the lower bound. The maximum degree test $\phi_n : \inbraces{0,1}^{N} \rightarrow \inbraces{0,1}$ is defined (for $t^*$ in \eqref{eqThresholdDefn}) by
\begin{align}
    \phi_n(G) = \1\!\inbraces{\Delta \geq t^*}
\end{align}
That is, $\phi_n$ returns $0$ (resp.~$1$) if the result is that $G \sim \PP_0$ (resp.~$G \sim \PP_1$). 

\begin{proposition}\label{propTV_lowerbound}
The Type-I and Type-II errors of the maximum degree test satisfy
\[
    \PP_0\inparen{\phi_n(G) = 1} = o(1), \qquad\text{and}\qquad \PP_1\inparen{\phi_n(G) = 0} = \Phi\Big(\frac{\gamma}{\sqrt{2}}\Big).
\]
Consequently, 
\[
\TV(\PP_0,\PP_1)\;\ge\; 1-\Phi\Big(\frac{\gamma}{\sqrt{2}}\Big)+o(1).
\]
\end{proposition}
Recall from~\eqref{eqMuSigDef} that $\mu$ and $\sigma^2$ denote the mean and variance of any vertex degree under $H_0$.  
\begin{lemma}\label{lemma:maxDegreeStatistics}
Under the planted distribution $G \sim \PP_1$, the degree of the planted center $\PSCenterVertex$ satisfies 
\begin{equation}\label{eq:planted-center-mean-var}
\E_1[d_{\PSCenterVertex}] \;=\; \mu + k + o(1),
\qquad
\Var_1(d_{\PSCenterVertex})\;=\;\sigma^2(1+o(1)).
\end{equation}
\end{lemma}
\begin{proof}
The degree of the planted hub vertex can be written as 
\begin{equation}\label{eq:planted-center-law}
d_{\PSCenterVertex} \;=\; k + H,\qquad H\sim \Hyp(N-k,\;n-1-k,\;m-k).
\end{equation}
The result then follows by direct computation.
\end{proof}
\begin{fact}\label{factTStar_and_PSCenter}
    In the critical scaling window \eqref{eqScalingWindowM}, the threshold $t^*$ in \eqref{eqThresholdDefn} satisfies
    \[
    t^* = \EE_1[d_{\PSCenterVertex}] + \frac{\gamma}{\sqrt{2}}\sigma + o(\sigma).
    \]
\end{fact}
\begin{proof}
Specializing to the critical window of $m$ in \eqref{eqScalingWindowM}, we have from \eqref{eqMuSigDef} that
\[
\mu = \frac{2m}{n}= \frac{k^2}{2\log n}\Big(1+\frac{\gamma}{\sqrt{\log n}}\Big),
\qquad
\sigma \sim \sqrt{\mu} = \frac{k}{\sqrt{2\log n}}
\Big(1+O\Big(\frac{1}{\sqrt{\log n}}\Big)\Big).
\]
We have the approximation
\begin{align}
\sigma\sqrt{2\log n}
&= \sqrt{2\mu\log n}
= k\sqrt{1+\frac{\gamma}{\sqrt{\log n}}}
= k + \frac{\gamma}{2}\frac{k}{\sqrt{\log n}} + o\Big(\frac{k}{\sqrt{\log n}}\Big).
\label{eq:max-expansion}
\end{align}
Thus, by a Taylor expansion,
\begin{align*}
t^* &\sim \frac{2m}{n} +  \sqrt{\frac{2m}{n}} \sqrt{2 \log n - \frac{\log \log n}{\alpha}} \sim \mu + \sigma \sqrt{2 \log n} \sqrt{1 - \frac{\log \log n}{2\alpha \log n}} \\
&= \mu + \sigma\sqrt{2 \log n} + \sigma\cdot  O\!\inparen{ \frac{\log \log n}{\log n}} \\
&= \mu + k + \frac{\gamma}{2} \frac{k}{\log n} + o(\sigma),
\end{align*}
and the result follows since $\sigma \sim k/\sqrt{2 \log n}$.
\end{proof}
\begin{proof}[Proof of Proposition \ref{propTV_lowerbound}]
By the asymptotic normality of (standardized) Hypergeometric r.v.'s~in the moderate regime Lemma \ref{lemHypTailToGaussianCDF}, and Fact \ref{factTStar_and_PSCenter}, we have
\[
\PP_1[\Delta \geq t^*] = \PP_1[d_{v_*} \geq t^*] = \PP_1\insquare{ \frac{d_{v_*} - \mu - k}{\sigma} \geq \frac{t^* - \mu - k}{\sigma}  } = 1-\Phi\!\Big(\frac{\gamma}{\sqrt{2}}\Big)+o(1),
\]
Furthermore, $\PP_0[\Delta \geq t^*] = o(1)$ by Lemma \ref{lem:maxdeg_short}. The result follows from $\TV(\PP_0, \PP_1) \geq \PP_1[\phi_n(G) = 1] + \PP_0[\phi_n(G) = 1]$.
\end{proof}

\subsection{An upper bound}

The main result of this section is the following.
\begin{lemma}\label{lemUpperBound}
    The total variation distance between $\PP_0$ and $\PP_1$ satisfies
    \begin{align*}
        \TV(\PP_0, \PP_1) \leq 1 - \Phi\bigg( \frac{\gamma}{\sqrt{2}} \bigg) + o(1).
    \end{align*}
\end{lemma}
The upper bound relies on the following lemma. 
\begin{lemma}\label{lemLR_upperBound}
Let $t^*$ be defined in \eqref{eqThresholdDefn} for any fixed $\alpha > 1$. Suppose that the event $\inbraces{\Delta \leq t^*}$ holds, and suppose that we are within the critical window \eqref{eqCriticalWindow_m_in_c} for $c = \tfrac14$. Then for any fixed $\varepsilon > 0$,
\[
\Lambda(G) \leq \frac{1}{2} + \epsilon, \qquad \text{whp under both $\PP_0$ and $\PP_1$}.
\]
\end{lemma}

\begin{proof}
    The assertion for the null distribution $\PP_0$ is immediate by Proposition \ref{propZphasetransition} and Lemma \ref{lem:maxdeg_short} which shows that the hypothesis event holds with high probability.

    It remains to prove the assertion for the planted distribution $\PP_1$. Recall that $\PSCenterVertex$ denotes the planted hub vertex. By Lemma \ref{lemLR_expansion} we may write, whp $\PP_1$,
    \begin{align}\label{eqUpperBound_LR_Decomp}
        \Lambda(G) = (1+o(1)) \frac{e^{-a_n^2/2}}{n}\inbraces{ e^{a_n Y_{\PSCenterVertex}}  + \sum_{i \in L} e^{a_n Y_i} + \sum_{i \in R} e^{a_n Y_i}  },
    \end{align}
    where we denote by $L$ and $R$ the set of star leaf vertices and the set of vertices not involved with the planted star respectively. We first show that the contribution from the planted center $\PSCenterVertex$ is small under the event $\inbraces{\Delta \leq t^*}$. Note that within the scaling window \eqref{eqScalingWindowM},
    \begin{align*}
        a_n = \frac{k\sigma}{\mu} = \sqrt{2 \log n} - \frac{\gamma}{\sqrt{2}} + o(1).
    \end{align*}
    It follows that
    \begin{align}\label{eqUpperBound_PSCenter_Bound}
        \frac{e^{-a_n^2/2}}{n} e^{a_n Y_{\PSCenterVertex}} &\leq \frac{e^{-a_n^2/2}}{n} \exp\inbraces{a_n \inparen{ \frac{t^*-\mu}{\sigma}  }}  \nonumber \\ 
        &= \exp\inbraces{  - \frac{a_n^2}{2} - \log n + \inparen{1 + O\!\inparen{\frac{1}{n}}} a_n \sqrt{2 \log n - \frac{\log \log n}{\alpha}}  }  \nonumber \\
        &= \exp \inbraces{ \frac{\gamma^2}{2} - \frac{\log \log n}{2\alpha} + o(1)   } = (1 + o(1)) e^{\frac{\gamma^2}{2}} (\log n)^{-2\alpha} = o(1).
    \end{align}
        For the non-center vertices, consider the following coupling between $\PP_1$ and a $\PP_0$. Let $\pi$ be a uniform random permutation over the set of all edges $E = \binom{[n]}{2}$. To generate the planted distribution, first pick a set of planted star edges $S$ out of $E$. Then choose the set of unplanted edges $H$ to be the first $m-k$ edges in $E \setminus S$ under this permutation $\pi$. The graph $G^{(1)} = S \cup H$ is a sample from $\PP_1$. A sample $G^{(0)}$ from $\PP_0 = G(n,m)$ is generated by choosing the first $m$ edges of $E$ under $\pi$.

        Observe that the first $m-k$ edges from $E \setminus S$ must occur within the first $m$ positions. Thus $H \subseteq G^{(0)}$ almost surely under this coupling. It follows that the degrees $\{d_i^{(1)}\}$ of the vertices $i$ involved in the edge set $H$ are related to the corresponding degrees $\{ d_i^{(0)} \}$ in $G^{(0)}$ by
        \[
        d_i^{(1)} \leq d_i^{(0)} + \1\!\inbraces{i \text{ is a star leaf in $G^{(1)}$}} \quad \text{a.s.}
        \]
        We may identify $L = V(S) \setminus \inbraces{v_*}$ and $R = [n]\setminus (L \cup \inbraces{v_*})$. Let $Y_i^{(b)} = (d_i^{(b)} - \mu)/\sigma$ for $b = 1,2$. By Proposition \ref{propZphasetransition} again, since $c = \tfrac{1}{4}$
        \begin{align}\label{eqUpperBound_nonPlant_Bound}
            \frac{e^{-a_n^2/2}}{n} \sum_{i \in R}   e^{a_n Y_i^{(1)}} \leq  \frac{e^{-a_n^2/2}}{n} \sum_{i \in R}   e^{a_n Y_i^{(0)}} \leq \frac{1}{2} + \varepsilon \quad \text{whp.}
        \end{align}
        On the other hand, we claim that
        \begin{align}\label{eqUpperBound_starLeaves_Bound}
            \frac{e^{-a_n^2/2}}{n} \sum_{i \in L}   e^{a_n Y_i^{(1)}} \leq e^{\frac{a_n}{\sigma}} \underbrace{\frac{e^{-a_n^2/2}}{n} \sum_{i \in L}   e^{a_n Y_i^{(0)}}}_{=: T^{(0)}} = o(1) \quad \text{whp.}
        \end{align}
        To see this, apply Markov's inequality and the Hypergeometric MGF bound Lemma \ref{lemHyp_Cum_CGF_MGF_Bound} so that for any $\delta > 0$, since $k \ll \sqrt{n}$,
        \begin{align*}
            \PP_0(T^{(0)} >\delta) \leq \frac{1+o(1)}{\delta}\frac{k}{n} \to 0 \quad \text{as $n \to \infty$}.
        \end{align*}
        It follows that $T^{(0)} = o(1)$ whp. We also have $\frac{a_n}{\sigma} \sim \frac{\log n}{k} = o(1)$ since $k \gg (\log n)^2$. 
        
        Plugging \eqref{eqUpperBound_PSCenter_Bound}, \eqref{eqUpperBound_nonPlant_Bound}, and \eqref{eqUpperBound_starLeaves_Bound} into \eqref{eqUpperBound_LR_Decomp} yields the result for $\PP_1$.
\end{proof}

\begin{proof}[Proof of Lemma \ref{lemUpperBound}]
Recall the threshold $t^*$ defined in \eqref{eqThresholdDefn}. For some $\epsilon > 0$, define the threshold $u_{\varepsilon} > t^*$ by 
\begin{align*}
u_{\varepsilon} = \frac{2m}{n}+\sqrt{\frac{2m}{n}\Big(1-\frac{m}{N}\Big)}\,
\sqrt{2\log n \inparen{ 1 + \frac{\varepsilon}{\sqrt{\log n}}  }-\frac{1}{\alpha}\log\log n}\, .
\end{align*}
This point $u_\epsilon$ is chosen such that $u_{\epsilon} = \mu + k + \sigma \frac{\gamma + \varepsilon}{\sqrt{2}} + o(\sigma)$, as can be verified by similar arguments as in Section \ref{secLowerBound}. Write $\mu_{b}$ for the probability mass function of $\PP_{b}$, $b = 1,2$. We have
\begin{align*}
\TV(\PP_0,\PP_1)
&=\sum_{\mu_1(G)>\mu_0(G)} \big(\mu_1(G)-\mu_0(G)\big)\\
&\le
\sum_{\substack{\mu_1>\mu_0,\\ \Delta\le t^*}} \big(\mu_1(G)-\mu_0(G)\big)
+
\sum_{\substack{\mu_1>\mu_0,\\ \Delta\ge u_{\varepsilon}}} \big(\mu_1(G)-\mu_0(G)\big)
\;+\;
\sum_{t^*\le \Delta\le u_{\varepsilon}} \mu_1(G).
\end{align*}
The third sum is $\le K\varepsilon$ for some constant $K > 0$ by the asymptotic normality of the standardized planted center degree $d_{v_*}$ and the fact that whp none of the other vertices will have degree exceeding $t^*$. The first sum is $\le \varepsilon$
since if $\Delta\le t^*$, we have $\mu_1(G)\le (\tfrac{1}{2}+\varepsilon)\mu_0(G)$ by Lemma \ref{lemLR_upperBound}. For the second sum,
\begin{align*}
\sum_{\substack{\mu_1>\mu_0,\\ \Delta\ge u_{\varepsilon}}} \big(\mu_1(G)-\mu_0(G)\big)
&= \sum_{\Delta\ge u_{\varepsilon}} \big(\mu_1(G)-\mu_0(G)\big)
-\sum_{\substack{\Delta\ge u_{\varepsilon},\\ \mu_0>\mu_1}} \big(\mu_1(G)-\mu_0(G)\big) \\
&\leq \sum_{\Delta\ge u_{\varepsilon}} \mu_1(G)
+\sum_{\Delta\ge u_{\varepsilon}} \mu_0(G).
\end{align*}
In the above display, the first sum is upper bounded by $1-\Phi\!\left(\frac{c}{\sqrt2}\right)+o(1)$ by the same conditioning argument as in the proof of Lemma \ref{lem:maxdeg_short} showing that the max degree in the planted model is attained by the planted center whp, and Lemma \ref{lemHypTailToGaussianCDF} giving the asymptotic Gaussian tails. The second sum is $o(1)$ by Theorem \ref{thm:max_deg_gnm}. The result follows since $\varepsilon$ is arbitrary.
\end{proof}
\begin{proof}[Proof of Theorem \ref{thmOptimalTest}]
Denote events $A= \{\Lambda(G) \ge1 \}$ and $B = \{ \Delta(G) \ge t^\ast \}$. For the null model, the probability of the symmetric difference $A \ominus B$ is
\[
    \PP_0[A \ominus B] \leq \PP_0[A] + \PP_0[B] = o(1)
\]
by Proposition \ref{propZphasetransition} and Lemma \ref{lem:maxdeg_short}. For the planted model, write
\[
\PP_1[A \ominus B] = \PP_1[A \mid B^c]\cdot \PP_1[B^c] + \PP_1[A^c \cap B]
\]
We have $\PP_1[A \mid B^c] = o(1)$ by Lemma \ref{lemLR_upperBound}. For the second term, for some $\epsilon > 0$
\begin{align*}
    &\PP_1[A^c \cap B] = \PP_1[\Lambda \leq 1 \mid \Delta \geq t^*] \\
    &= \PP_1[\Lambda \leq 1 \mid t^* \leq \Delta \leq t^* + \epsilon \sigma] \cdot \PP_1[ \Delta \leq t^*+ \epsilon \sigma \mid \Delta \geq t^* ] \\
    &\qquad + \PP_1[\Lambda \leq 1 \mid \Delta \geq t^* + \epsilon \sigma] \cdot \PP_1[ \Delta \geq t^* + \epsilon \sigma \mid \Delta \geq t^*].
\end{align*}
For the first term above, we have 
\[
\PP_1[ \Delta \leq t^*+ \epsilon \sigma \mid \Delta \geq t^* ]  = \frac{\PP_1[ t^* \leq \Delta \leq t^* + \epsilon \sigma  ] }{\PP_1[\Delta \geq t^*]} = O(\epsilon)
\]
which follows from Lemma \ref{lem:maxdeg_short} and Lemma \ref{lemHypTailToGaussianCDF}. For the second term above, we note that under the event $\inbraces{\Delta \geq t^* + \epsilon \sigma}$, with $a_n = (1+o(1))\sqrt{2 \log n}$ when $c = \tfrac{1}{4}$,
\begin{align*}
\Lambda(G) &= (1+o(1)) \frac{e^{-a_n^2/2}}{n}\inbraces{  \sum_{i =1}^{n} e^{a_n Y_i}  } \geq (1+o(1)) \frac{e^{-a_n^2/2}}{n} e^{a_n Y_{d_{\PSCenterVertex}}} \\
&\geq (1+o(1))\exp\inbraces{ - 2\log n + \sqrt{2 \log n} \inparen{ \sqrt{ 2 \log n - \frac{1}{\alpha}\log \log n  }  + \epsilon } } \\
&= (1+o(1))\exp\inbraces{ - 2\log n + \sqrt{2 \log n} \inparen{ \sqrt{ 2 \log n \inparen{1  - \frac{\log \log n}{2 \alpha \log n} } }  + \epsilon } } \\
&= (1+o(1)) \exp \inbraces{ \sqrt{2 \log n} \epsilon + O\!\inparen{\frac{\log \log n}{\log n}}  }\,, 
\end{align*}
and so $\PP_1[\Lambda \leq 1 \mid \Delta \geq t^* + \epsilon \sigma] = 0$ if $\eps = \omega( (\log n)^{-1/2})$.  Taking $\epsilon \to 0$ sufficiently slowly finishes the proof.
\end{proof}

\section{Maximum degrees in the null and planted models}\label{secMaxDegreesInNullAndPlanted}

In this section we prove Lemma \ref{lem:maxdeg_short}. The following implies the assertion for the null model.

\begin{lemma}\label{thm:max_deg_gnm}
Let $G\sim \PP_0 = G(n,m)$ and suppose $n (\log n)^3 \ll m \ll n^2$. For fixed $\alpha>0$ let $t^*$ be defined as in \eqref{eqThresholdDefn}.
Then
\[
\PP\big(\Delta \ge t^*\big)=o(1)\quad\text{if }\alpha>1,
\quad\text{and}\quad
\PP\big(\Delta \ge t^*\big)=1-o(1)\quad\text{if }\, 0<\alpha<1.
\]
\end{lemma}
Observe that the property $Q = \{\Delta(G) \geq t^*\}$ is convex: in the sense that if $G_1$ and $G_2$ satisfy $Q$, then so must every $G$ satisfying $G_1 \subseteq G \subseteq G_2.$ By similar reasoning, the property $\{\Delta(G) \leq k_\alpha\}$ is also convex. Therefore, by criterion (ii) of Theorem 2.2 in \cite{bollobas2001random}, it suffices to prove the following result for the \ER~$G(n,p)$ model.

\begin{lemma} \label{lem:max_deg_gnp}
Let $G \sim G(n,p)$ and suppose that $(n-1)p(1-p) \gg (\log n)^3$. For fixed $\alpha>0$ let $t^*$ be defined as in \eqref{eqThresholdDefn}. Then
\[
\PP\big(\Delta \ge t^*\big)=o(1)\quad\text{if }\alpha>1,
\qquad\text{and}\qquad
\PP\big(\Delta \ge t^*\big)=1-o(1)\quad\text{if }\,0<\alpha<1.
\]
\end{lemma}
An advantage of the $G(n,p)$ model is that we are able to apply the following DeMoivre-Laplace theorem, which approximates the Binomial tail by a Gaussian.
\begin{lemma}[DeMoivre–Laplace (see e.g.~\cite{raab1998balls} Theorem 2)] \label{lem:dem_laplace} Suppose that $p = p_n \in [0,1]$ satisfies  $n p(1 - p)  \to \infty$ as  $n \to \infty$. If $0 < h = x (np (1-p) )^{1/2} = o((np(1-p))^{2/3})$ and $x \to \infty$, then
\[
\PP[\Bin(n, p) \geq n p + h] = \left(1 + o(1)\right) \cdot \frac{1}{x \sqrt{2 \pi}} \exp\left(-\frac{x^2}{2}\right).
\]
\end{lemma}

\begin{proof}[Proof of Lemma~\ref{lem:max_deg_gnp}] Define $
x_{n,\alpha}^2 = 2\log n - \frac{1}{\alpha}\log\log n$ and set $t_{n,\alpha}
=
(n-1)p \;+\; \sqrt{(n-1)pq}\,x_{n,\alpha}$, where $q = 1-p$. Our proof will proceed via an application of the second moment method, much like that of \cite{raab1998balls}.
For each $v\in[n]$, we have $d_v\sim\Bin(n-1,p)$. Set
\[
\mu_n:=(n-1)p,\qquad \sigma_n^2:=(n-1)pq,
\]
so that $t_{n,\alpha}=\mu_n+x_{n,\alpha}\sigma_n$. Define
\[
X_{n,\alpha}:=\sum_{v=1}^n \mathbf{1}\{d_v \ge t_{n,\alpha}\}.
\]
Then $\{\Delta \ge t_{n,\alpha}\}=\{X_{n,\alpha}\ge 1\}$. Let $n':=n-1$. We will apply Lemma~\ref{lem:dem_laplace} with $h:=x\sqrt{n'pq}$.
The condition $h=o((n'pq)^{2/3})$ is equivalent to $x=o((n'pq)^{1/6})$,
so 
\begin{equation}\label{eq:EX}
\E[X_{n,\alpha}]
=
n\,\PP(d_1\ge t_{n,\alpha})
=
(1+o(1))\,n\cdot \frac{1}{x_{n,\alpha}\sqrt{2\pi}}e^{-x_{n,\alpha}^2/2}.
\end{equation}
Taking logs gives
\begin{align}
\log \E[X_{n,\alpha}]
&=
\log n - \frac{x_{n,\alpha}^2}{2} - \log x_{n,\alpha} - \frac12\log(2\pi) + o(1)\notag\\
&=
\log n - \Big(\log n - \frac{1}{2\alpha}\log\log n\Big)
      - \log x_{n,\alpha} - \frac12\log(2\pi) + o(1)\notag\\
&=
\frac{1}{2\alpha}\log\log n - \log x_{n,\alpha} + O(1) + o(1).\label{eq:logEX-general}
\end{align}
Since $x_{n,\alpha}^2 = 2\log n\,(1+o(1))$, we have $\log x_{n,\alpha}=\frac12\log\log n+O(1)$, and therefore
\begin{equation}\label{eq:EX-asymp}
\E[X_{n,\alpha}]
=
(\log n)^{\frac{1}{2\alpha}-\frac12+o(1)}.
\end{equation}
If $\alpha>1$, then $\E[X_{n,\alpha}]=o(1)$, and Markov's inequality gives
\[
\PP(\Delta\ge t_{n,\alpha})=\PP(X_{n,\alpha}\ge 1)\le \E[X_{n,\alpha}]=o(1).
\]
If $0<\alpha<1$, then $\E[X_{n,\alpha}]\to\infty$. Thus, if we can show that
\begin{equation}\label{eq:second-moment-goal}
\E[X_{n,\alpha}^2] = (1+o(1))\,\E[X_{n,\alpha}]^2,
\end{equation}
then the theorem follows from the second moment method. 
Write $I_v:=\mathbf{1}\{d_v\ge t_{n,\alpha}\}$. Then $X_{n,\alpha}=\sum_{v=1}^n I_v$ and
\[
\E[X_{n,\alpha}^2]=\E[X_{n,\alpha}] + n(n-1)\E[I_1I_2].
\]
Condition on the edge $\{1,2\}$. Let $Z:=\mathbf{1}\{\{1,2\}\in E\}\sim\Bern(p)$ and
let $B\sim\Bin(n-2,p)$. Conditional on $Z$, the variables $d_1-Z$ and $d_2-Z$ are independent
and each distributed as $B$, hence
\[
\E[I_1 I_2]
=
p\,\PP(B\ge t_{n,\alpha}-1)^2 + (1-p)\,\PP(B\ge t_{n,\alpha})^2.
\]
Since $x_{n,\alpha}\to\infty$ and $x_{n,\alpha}=o(((n-2)pq)^{1/6})$, Lemma~\ref{lem:dem_laplace}
applies uniformly at levels $x_{n,\alpha}+o(1)$. Because shifting the threshold by $1$ changes the standardized level by
$O(\sigma_n^{-1})=o(1)$, we have
\[
\PP(B\ge t_{n,\alpha}-1)=(1+o(1))\PP(d_1\ge t_{n,\alpha}),
\qquad
\PP(B\ge t_{n,\alpha})=(1+o(1))\PP(d_1\ge t_{n,\alpha}).
\]
Therefore $\E[I_1 I_2]=(1+o(1))\PP(d_1\ge t_{n,\alpha})^2$, and hence
\[
\E[X_{n,\alpha}^2]
=
\E[X_{n,\alpha}] + (1+o(1))\,n(n-1)\PP(d_1\ge t_{n,\alpha})^2
=
\E[X_{n,\alpha}] + (1+o(1))\,\E[X_{n,\alpha}]^2
\]
as desired.
\end{proof}

\begin{proof}[Proof of Lemma \ref{lem:maxdeg_short}]
The first claim follows directly from Lemma \ref{thm:max_deg_gnm}. For the second claim, write $v_*$ for the planted center and let
$H$ denote the subgraph induced by $[n]\setminus\{v_*\}$. By construction of the planted model, $H$ contains at most $m - k$ edges. Hence, for every increasing event $\mathcal E$
depending only on $H$, we have the stochastic domination bound
\[
\P_1\big(\mathcal E(H)\big)
\;\le\;
\P\big(\mathcal E(G(n-1,m))\big).
\]
Moreover, for each $u\neq v_*$ we can decompose
\[
d_u \;=\; d_u(H) + \mathbf 1\{(u,v_*)\in E\},
\]
so in particular $d_u \le d_u(H)+1$ and therefore
\[
\max_{u\neq v_*} d_u
\;\le\;
\Delta(H)+1.
\]
It follows that
\begin{align*}
\P_1\big(\Delta \ge t^*,\ d_{v_*}< t^*\big)
&\le
\P_1\big(\exists\,u\neq v_*:\ d_u\ge t^*\big) \\
&\le
\P_1\big(\Delta(H)\ge t^*-1\big) \\
&\le
\P\big(\Delta(G(n-1,m))\ge t^*-1\big)
\;=\; o(1),
\end{align*}
where the last step uses Lemma \ref{thm:max_deg_gnm} (applied with $n-1$ in place of $n$),
together with the fact that replacing $t^*$ by $t^*-1$ does not change the $o(1)$
conclusion. Finally, since $\{d_{v_*}\ge t^*\}\subseteq \{\Delta\ge t^*\}$, we have
\[
\P_1(\Delta\ge t^*)
=
\P_1(d_{v_*}\ge t^*)
+
\P_1(\Delta\ge t^*,\ d_{v_*}< t^*)
=
\P_1(d_{v_*}\ge t^*)+o(1).
\qedhere \]
\end{proof}

 \section*{Acknowledgments}
WP supported in part by  NSF grant CCF-2309708.

\bibliography{ref}
\bibliographystyle{abbrv}

\appendix

\section{Further properties of the Hypergeometric distribution}\label{secApproxHypergeom}

In this Appendix we collect relevant technical results about the Hypergeometric distribution.

\subsection{Strongly Rayleigh property}
\label{secHypergeomStronglyRayleigh}

For a random variable $Y$, define its probability generating function (PGF) to be $\lambda \mapsto \EE[\lambda^Y]$ for $\lambda \in \CC$. We say that $Y$ has a \emph{strongly Rayleigh} distribution (c.f.~\cite{BorceaBrandenLiggett2009Negative} Defintions 2.9 and 2.10) if its PGF is \emph{real stable} i.e.~it has real coefficients and for $\Im(\lambda) > 0$, $\EE[\lambda^Y] \neq 0$.

\begin{fact}[Essentially in \cite{BorceaBrandenLiggett2009Negative}]\label{fact}
    The Hypergeometric distribution is strongly Rayleigh.
\end{fact}
\begin{proof}
    Consider the uniform distribution over subsets of size $m$ of $[N]$, i.e.~$\PP[\bX = \1_{S}] = \binom{N}{m}^{-1}$ for any $S \subset [N]$ with $\abs{S} = m$, where $\bX = (X_1,\dots,X_N$) is a collection of $\{0,1 \}$-valued random variables with $X_i$ indicating inclusion of the $i$-th element. Its PGF is
    \[
    g_{\bX}(z_1,\dots,z_N) = \frac{1}{\binom{N}{m}} e_m(z_1,\dots,z_N),
    \]
    where $e_m(z_1,\dots,z_N) = \sum_{\substack{S \subseteq [N] \\ \abs{S} = m}} \prod_{i \in S} z_i$ is the $m$-th elementary symmetric polynomial. Define $F(t,z_1,\dots,z_N) = \prod_{i=1}^{N}(t+z_i) = \sum_{m=0}^{N} e_m(z_1,\dots,z_N) t^{N-m}$. Note that $F(t,z_1,\dots,z_N)$ is real stable. On the other hand $e_m(z_1,\dots,z_N) = (N-m)!\partial_t^{N-m} F(t,z_1,\dots,z_N) \rvert_{t = 0}$. By \cite[Proposition 3.1]{BorceaBrandenLiggett2009Negative}, real stability is closed under taking derivatives and specializing to real constants, thus $e_m$ inherits the real stability property from $F(t,z_1,\dots,z_N)$.

    Suppose without loss of generality that the first $K$ indicators $X_1,\dots,X_K$ are marked as desirable objects. Then $Y = \sum_{i=1}^{K} X_i \sim \Hyp(N,K,m)$ (see Definition \ref{defHypergeometric}). Let $g_Y(\lambda) = \EE \lambda^Y$ denote the PGF of $Y$. We have the relation
    \[
    g_Y(\lambda) = g_{\bX}(\underbrace{\lambda,\dots,\lambda}_{K \text{ times}}, 1,\dots,1).
    \]
    By \cite[Proposition 3.1, Definition 2.8]{BorceaBrandenLiggett2009Negative}, since $g_Y$ is the projection and diagonalization of $g_{\bX}$ onto the first $K$-coordinates, it is a real stable polynomial.
\end{proof}
A univariate PGF is real stable if and only if it has all real roots. Coupled with the fact that the Hypergeometric PGF has nonnegative coefficients, we arrive at the following.
\begin{corollary}\label{corrHypStrictlyNegativeRoots}
    Let $Y \sim \Hyp(N,K,m)$ where $0 < K < N$ and $m > 0$. Then the PGF of $Y$ has strictly negative roots.
\end{corollary}

\subsection{Gaussian approximation for Hypergeometric}

In this section we derive some relationships between the Hypergeometric tails and those of the Gaussian. We begin with a bound on cumulants which leads to asymptotics of the moment and cumulant generating functions. For any random variable $Y$, we denote by $\kappa_r(Y)$ its $r^{th}$ cumulant.
\begin{lemma}\label{lemHyp_Cum_CGF_MGF_Bound}
Let $d_x \sim \Hyp(N,n-1,m)$ be the degree of vertex $x$ in $G(n,m)$ and set $Y := \frac{d_x - \mu}{\sigma}$, where $\mu$ and $\sigma$ are the mean and variance of $d_x$ given in \eqref{eqMuSigDef}. 
Then there exists a universal constant $C_0 > 0$ such that
\begin{equation}\label{eqHypCumulantBound}
    \abs{\kappa_r(Y)} \leq \frac{C_0^r r!}{\sigma^{r-2}}, \qquad \text{for all $r \geq 3$}.
\end{equation}
Furthermore, for $|t|=o(\sigma^{1/3})$ we have $\log \E[e^{tY}]=\frac{t^2}{2}+o(1)$ and $\E[e^{tY}]=(1+o(1))e^{t^2/2}$.
\end{lemma}

\begin{proof}
Our first step is establish the following bound for the cumulants:
\begin{equation}\label{eqHypCumulantGoalInStirling}
|\kappa_r(Y)|
\leq \frac{2}{\sigma^{r-2}} \sum_{q=1}^{r} (q-1)!\,S(r,q)
\end{equation}
where $S(r,q)$ denotes Stirling numbers of the second kind. To this end, note that the Hypergeometric probability generating function (PGF) for the degree $d_x$ is
\begin{equation}\label{eq:deg-hyp-pgf}
g(z):=\E[z^{d_x}]
=\sum_{d=0}^{n-1}\P(d_x=d)\,z^{d}
=\frac{1}{\binom{N}{m}}
\sum_{d=0}^{n-1}
\binom{n-1}{d}\binom{N-(n-1)}{m-d}\,z^{d}.
\end{equation}
By Corollary \ref{corrHypStrictlyNegativeRoots}, all the zeroes of $g$ are real and strictly negative.
Thus, with $\{\zeta_j\}$ denoting these negative real roots, we have the factorization
\begin{align}\label{eq:P-root-factor}
g(z)&=C\prod_{j=1}^{m}(z-\zeta_j) = \prod_{j=1}^{m}\frac{z-\zeta_j}{1-\zeta_j},
\end{align}
where $
C :=\prod_{j=1}^{m}(1-\zeta_j)^{-1}$ which is deduced from $g(1)=1$. Furthermore, since $\zeta_j<0$, by setting $a_j:=-1/\zeta_j>0$, we have
\[
\frac{z-\zeta_j}{1-\zeta_j}
=\frac{1+a_j z}{1+a_j},
\]
and \eqref{eq:P-root-factor} becomes
\begin{equation}\label{eq:P-az-factor}
g(z)=\prod_{j=1}^{m}\frac{1+a_j z}{1+a_j}.
\end{equation}
Define the moment and cumulant generating functions $M(t):=\E[e^{td_x}]$ and
$K(t):=\log M(t)$ respectively. Since $d_x$ is integer-valued, we have the relations
\[
M(t)=g(e^t),\qquad K(t)=\log g(e^t),\qquad\text{and}\qquad \kappa_r(d_x)=K^{(r)}(0).
\]
Plugging \eqref{eq:P-az-factor} into $K(t)$ yields
\begin{equation}\label{eq:K-sum-phi}
K(t)=\log g(e^t)
=\sum_{j=1}^{m}\Big(\log(1+a_j e^t)-\log(1+a_j)\Big)
=\sum_{j=1}^{m}\phi_{a_j}(t),
\end{equation}
where $
\phi_a(t):=\log(1+a e^t)-\log(1+a).$
For $a>0$, define
\[
\rho:=\frac{a}{1+a}\in(0,1),
\]
so that $
\phi_a(t)=\log\big(1 + \rho(e^{t} - 1)\big).$ We have the expansions
\[
\log(1+u)=\sum_{q\ge1}(-1)^{q+1}\frac{u^q}{q}, \quad\text{and}\quad (e^t-1)^q=q!\sum_{\ell\ge q} S(\ell,q)\frac{t^\ell}{\ell!},
\]
where $S(\ell,q)$ is the Stirling number of the second kind. It follows that
\begin{align*}
\phi_a(t)
&=\sum_{q\ge 1}(-1)^{q+1}\frac{\rho^q}{q}
\left(q!\sum_{\ell\ge q} S(\ell,q)\frac{t^\ell}{\ell!}\right) =\sum_{q\ge 1}(-1)^{q+1}\rho^q (q-1)!
\sum_{\ell\ge q} S(\ell,q)\frac{t^\ell}{\ell!}.
\end{align*}
Now, fix an integer $r\ge 3$. Reading off the coefficient of $t^r/r!$, we obtain the formula
\begin{equation}\label{eq:phi-deriv-stirling}
\phi_a^{(r)}(0)=\sum_{q=1}^{r}(-1)^{q+1}\rho^q (q-1)!\,S(r,q).
\end{equation}
Now, observe that $\phi_a''(0) = \rho(1-\rho)$. So, we can write
\begin{align}\label{eq:case1}
|\phi_a^{(r)}(0)|
&\le \rho \sum_{q=1}^{r} (q-1)!\,S(r,q) \le 2\,\rho(1-\rho)\sum_{q=1}^{r} (q-1)!\,S(r,q) = 2\phi_a''(0)\sum_{q=1}^{r} (q-1)!\,S(r,q)
\end{align}
Putting everything together,
\begin{align}
|\kappa_r(d_x)|
\le \sum_{j=1}^{m}|\phi_{a_j}^{(r)}(0)| \le 2\Big(\sum_{q=1}^{r} (q-1)!\,S(r,q)\Big)\sum_{j=1}^{m}\phi_{a_j}''(0)
= 2\sigma^2\Big(\sum_{q=1}^{r} (q-1)!\,S(r,q)\Big) 
\end{align}
Because $Y=(d_x-\mu)/\sigma$, we have $\kappa_r(Y)=\kappa_r(d_x)/\sigma^r$, so the desired \eqref{eqHypCumulantGoalInStirling} follows for all $r \geq 3$.

Next consider the cumulant generating function $K_Y(t)$ of $Y$:
\begin{equation}\label{eq:split-series}
K_Y(t)=\frac{t^2}{2} + \sum_{r\ge 3}\kappa_r(Y)\frac{t^r}{r!}.
\end{equation}
For notational convenience, define
\[
A_r:=\sum_{q=1}^r (q-1)!\,S(r,q) 
\]
so that \eqref{eqHypCumulantGoalInStirling} is $\abs{\kappa_r(Y)} \leq 2 A_r / \sigma^{r-2}$. Consider the exponential generating function
\[
F(z):=\sum_{r\ge 1} A_r \frac{z^r}{r!}.
\]
Using the identity $(e^z-1)^q=q!\sum_{r\ge q} S(r,q)\,z^r/r!$ and swapping sums,
\begin{align*}
F(z)
&=\sum_{r\ge 1}\sum_{q=1}^r (q-1)!\,S(r,q)\,\frac{z^r}{r!}
=\sum_{q\ge 1}(q-1)!\sum_{r\ge q}S(r,q)\frac{z^r}{r!} \\
&=\sum_{q\ge 1}(q-1)!\cdot \frac{(e^z-1)^q}{q!}
=\sum_{q\ge 1}\frac{(e^z-1)^q}{q}
=-\log\!\bigl(1-(e^z-1)\bigr) \\
&=-\log\!\bigl(2-e^z\bigr).
\end{align*}
The function $F(z)=-\log(2-e^z)$ is analytic on the disk $|z|<\log 2$, since the nearest
singularity occurs at  $z=\log 2$ and its complex translates.
Therefore the radius of convergence of $F$ is $R=\log 2$. Thus, fix any $\rho\in(0,\log 2)$ and set $M_\rho:=\sup_{|z|=\rho}|F(z)|<\infty$, where
$F(z)=\sum_{r\ge 1}A_r z^r/r!=-\log(2-e^z)$. By Cauchy's estimate,
\[
\frac{A_r}{r!}\le \frac{M_\rho}{\rho^r}\qquad\text{for all }r\ge 1,
\]
which proves the asserted bound on cumulants. Furthermore,
\begin{align*}
\Big|K_Y(t)-\frac{t^2}{2}\Big|
&\le |t|^2\sum_{r\ge 3}\frac{2A_r}{r!}\Big(\frac{|t|}{\sigma}\Big)^{r-2} 
\le 2M_\rho |t|^2 \sum_{r\ge 3}\rho^{-r}\Big(\frac{|t|}{\sigma}\Big)^{r-2} \\
&= 2M_\rho |t|^2 \rho^{-2}\sum_{k\ge 1}\Big(\frac{|t|}{\rho\sigma}\Big)^k.
\end{align*}
Since $|t|=o(\sigma^\frac{1}{3})$, we have $\frac{|t|}{\rho\sigma}<1$ and so
\begin{align}
\Big|K_Y(t)-\frac{t^2}{2}\Big|
\le
\frac{2M_\rho}{\rho^{3}}\,
\frac{|t|^{3}/\sigma}{1-\frac{|t|}{\rho\sigma}} = o(1)
\end{align}
Exponentiating gives
\[
M_Y(t)=\E[e^{tY}]
=\exp\!\Big(\frac{t^2}{2}+o(1)\Big)
=\exp\!\Big(\frac{t^2}{2}\Big)\,(1+o(1)).
\]
This finishes the proof.
\end{proof}
\begin{lemma}\label{lemHypTailToGaussianCDF}
    Let $d_x$ denote the degree of some arbitrary vertex $x$ in $\PP_0 = G(n,m)$. Define the standardized version $Y = (d_x-\mu)/\sigma$ and call its distribution $\QQ$. For any $a = o(\sigma^{1/3})$, define the tilted distribution $\QQ^{(a)}$ which has density $y \mapsto e^{ay}/\EE[e^{aY}]$ with respect to $\QQ$. Then for any $t = \Theta(a)$,
    \begin{align*}
        \QQ^{(a)}[Y \geq t] = (1+o(1)) \Phi^c(t - a).
    \end{align*}
\end{lemma}
\begin{proof}
    Let $\psi(a) = \log \EE[e^{aY}]$. Note from the cumulant bound \eqref{eqHypCumulantBound} that $\psi$ is analytic on $\abs{a} \leq c\sigma$ for some constant $c > 0$. Let $m_a$ and $s^2_a$ be the mean and variance of $Y \sim \QQ^{(a)}$ and we have, uniformly in $\abs{a} \leq c\sigma$,
    \[
    m_a = \psi'(a) = a + O\!\inparen{\frac{a^2}{\sigma}}, \qquad s_a^2 = \psi''(a) = 1 + O\!\inparen{ \frac{a}{\sigma} }.
    \]
    Define $Z_a = \frac{Y-m_a}{s_a}$. The cumulant bounds in \eqref{eqHypCumulantBound} transfer over to $Z_a$ so that uniformly in $\abs{a} \leq c\sigma$, we have the bounds $\abs{\kappa_r(Z_a)} \leq (r-2)!/(C\sigma)^{r-2}$ for some constant $C > 0$. Theorem 2.1 from \cite{DoringJansenSchubert2022Survey}, which originates from \cite{RudzkisSaulisStatuljavicus1978}, yields for all $t = o(\sigma)$,
    \[
    \QQ^{(a)}[Y \geq t] = (1+o(1)) \exp\inparen{O\!\inparen{\frac{(t-a)^3}{\sigma}}} \Phi^c\!\inparen{ \frac{t - m_a}{s_a}  }.
    \]
    The result follows since $m_a = a + o(1)$, $s_a = 1 + o(1)$ and $t-a = o(\sigma^{1/3})$.
\end{proof}

\subsection{Near-independence of degrees in \texorpdfstring{$G(n,m)$}{G(n,m)}}

In this section, we will formalize the extent to which the degree random variables in $G(n,m)$ are \emph{almost} independent. In particular, we will show that the moment generating functions of the degrees are asymptotically unchanged even after conditioning on some other vertex. Throughout, all probabilities $\PP$ and expectations $\EE$ are taken with respect to the null model $\PP_0$.

\begin{lemma}
\label{lem:mgf-stab-d2}
Let $G\sim \PP_0 = G(n,m)$, let $N:=\binom{n}{2}$, and assume $m\le (1-\delta)N$ for some fixed
$\delta\in(0,1)$. Let $\mu:=\E[d_1]=2m/n$ and let $\lambda=\lambda_n$ satisfy $|\lambda|\to 0$.
Then uniformly over all integers $k \in [n-1]$, we have
\[
\E\!\left[e^{\lambda(d_1-\mu)}\,\middle|\, d_2=k\right]
=\E\!\left[e^{\lambda(d_1-\mu)}\right]\cdot (1+o(1)).
\]
\end{lemma}

\begin{proof}
Write $B:=\1\{\{1,2\}\in E(G)\}$. Conditional on $d_2=k$, the $k$ incident edges to vertex $2$
form a uniformly random $k$-subset of the $(n-1)$ edges incident to $2$, and the remaining
$m-k$ edges are a uniformly random $(m-k)$-subset of the $N_0:=N-(n-1)$ edges not incident
to $2$; moreover these two random subsets are independent. In particular, conditional on $d_2=k$, we have
\[
\PP(B=1\mid d_2=k)=\frac{k}{n-1},
\qquad
B \perp\!\!\!\perp \Big\{\text{edges not incident to }2\Big\} \, \text{ given }d_2=k.
\]
Let
\[
X_k:=\#\{\{1,j\}\in E(G): j\ge 3\}.
\]
Given $d_2=k$, the random variable $X_k$ counts how many of the $K_0:=n-2$ edges
$\{1,j\}$ with $j\ge 3$ are selected among the $m-k$ edges chosen from the $N_0$ edges not
incident to $2$. Hence
\[
X_k\mid (d_2=k)\ \sim\ \Hyp(N_0,K_0,m-k),
\qquad\text{and}\qquad
X_k \perp\!\!\!\perp B\ \text{ given }d_2=k,
\]
and therefore $d_1=B+X_k$. Thus, for every admissible $k$,
\begin{align}
\label{eq:Fk-def}
\E\!\left[e^{\lambda(d_1-\mu)} \,\middle|\, d_2=k\right]
&= e^{-\lambda\mu}\,
\E\!\left[e^{\lambda B} \,\middle|\, d_2=k\right]\,
\E\!\left[e^{\lambda X_k} \,\middle|\, d_2=k\right] \\
&= e^{-\lambda\mu}\,
\Big(1+(e^\lambda-1)\frac{k}{n-1}\Big)\,M_0(m-k).
\end{align}
where for $s\in\{0,1,\dots,N_0\}$ we define
\[
M_0(s):= \E\!\left[e^{\lambda X_s}\right],
\quad X_s\sim \Hyp(N_0,K_0,s).
\]
We will now show that the right-hand side of \eqref{eq:Fk-def} is uniformly stable in $k$ up to a $(1+o(1))$ factor. Our strategy will be to show stability of $M_0$ and stability of the prefactor. Putting this together will give us $\log$-Lipschitz control of the conditional MGF, demonstrating that changes in $k$ do not affect the MGF of $d_2$ significantly, thus yielding the result.

Thus, we fix $s\ge 1$ and consider sampling without replacement from a population of size $N_0$ with
$K_0$ ``successes''. Expose the $s$ draws sequentially and let $X_t$ be the number of
successes in the first $t$ draws. Then
\[
M_0(s)=\E[e^{\lambda X_s}]
=\E\!\left[e^{\lambda X_{s-1}}\,\E[e^{\lambda \Delta_s}\mid \mathcal F_{s-1}]\right],
\quad \Delta_s:=X_s-X_{s-1}\in\{0,1\},
\]
where $\mathcal F_{s-1}$ is the history up to time $s-1$. Conditionally on $\mathcal F_{s-1}$,
the success probability at step $s$ is
\[
p_s:=\PP(\Delta_s=1\mid \mathcal F_{s-1})
=\frac{K_0-X_{s-1}}{N_0-(s-1)}\le \frac{K_0}{N_0-(s-1)}.
\]
Hence
\[
\E[e^{\lambda \Delta_s}\mid \mathcal F_{s-1}]
=1+(e^\lambda-1)p_s
\in \big[\,1+(e^\lambda-1)\cdot 0,\ \ 1+(e^\lambda-1)\tfrac{K_0}{N_0-(s-1)}\,\big].
\]
Since $m\le (1-\delta)N$ and $N_0=N-(n-1)$, we have $s\le m \le (1-\delta)N\le (1-\delta/2)N_0$
for all large $n$, so $N_0-(s-1)\ge (\delta/2)N_0$. With $K_0=n-2$ and $N_0=\Theta(n^2)$ this gives
\[
\frac{K_0}{N_0-(s-1)}\le \frac{C}{n}
\quad\text{for a constant }C=C(\delta).
\]
Therefore, for all large $n$ and all $1\le s\le m$,
\[
\Big|\log\frac{M_0(s)}{M_0(s-1)}\Big|
\le \Big|\log\Big(1+(e^\lambda-1)\frac{C}{n}\Big)\Big|
\le \frac{C'}{n}|e^\lambda-1|
\le \frac{C''}{n}|\lambda|,
\]
where we used $|\lambda|\to 0$ so $|e^\lambda-1|\le 2|\lambda|$ eventually. Thus
\begin{equation}
\label{eq:M0-step}
\sup_{1\le s\le m}\Big|\log M_0(s)-\log M_0(s-1)\Big|\ \le\ C_1\frac{|\lambda|}{n}.
\end{equation}
For the prefactor $1+(e^\lambda-1)\frac{k}{n-1}$, define $A(k):=1+(e^\lambda-1)\frac{k}{n-1}$. Then for all $k$,
\[
\Big|\log\frac{A(k+1)}{A(k)}\Big|
=\Big|\log\Big(1+\frac{e^\lambda-1}{(n-1)A(k)}\Big)\Big|
\le \frac{|e^\lambda-1|}{n-1}
\le C_2\frac{|\lambda|}{n}.
\]
Hence
\begin{equation}
\label{eq:A-step}
\sup_{0\le k\le n-2}\Big|\log A(k+1)-\log A(k)\Big|\ \le\ C_2\frac{|\lambda|}{n}.
\end{equation}
Now, denote $
F(k):=\E\!\left[e^{\lambda(d_1-\mu)}\mid d_2=k\right].
$
By \eqref{eq:Fk-def}, $\log F(k)= -\lambda\mu + \log A(k) + \log M_0(m-k)$.
Using \eqref{eq:M0-step} and \eqref{eq:A-step}, we have that uniformly in $k$,
\begin{align}
&\big|\log F(k+1)-\log F(k)\big| \nonumber
\\
&\leq \big|\log A(k+1)-\log A(k)\big| + \big|\log M_0(m-k-1)-\log M_0(m-k)\big| 
\le C_3\frac{|\lambda|}{n}
\end{align}
Summing this bound over at most $n$ steps yields
\[
\sup_{k, k' \in [n-1]}\Big|\log\frac{F(k)}{F(k')}\Big|
\le C_3|\lambda|=o(1),
\]
and thus
\begin{equation}
\label{eq:F-uniform}
\sup_{k \in [n-1]}\left|\frac{F(k)}{F(k_0)}-1\right|=o(1)
\end{equation}
We finish by noting that
\[
\E\!\left[e^{\lambda(d_1-\mu)}\right]=\sum_{k}\PP(d_2=k)\,F(k),
\]
so $\E[e^{\lambda(d_1-\mu)}]$ lies between $\min_k F(k)$ and $\max_k F(k)$ over admissible $k$.
By \eqref{eq:F-uniform}, $\max_k F(k)=(1+o(1))\min_k F(k)$, and therefore 
\[
F(k)=\E\!\left[e^{\lambda(d_1-\mu)}\right]\cdot (1+o(1))
\]
as desired. \end{proof}
\begin{corollary}
\label{cor:cond-trunc-allk}
Let $G\sim \PP_0 = G(n,m)$, let $N:=\binom{n}{2}$, and assume $m\le (1-\delta)N$ for some fixed
$\delta\in(0,1)$. Let $\mu:=\E[d_1]=2m/n$ and let $\lambda=\lambda_n$ satisfy $|\lambda|\to 0$.
Then, uniformly for every $k\in\{0,1,\dots,n-1\}$,
\begin{equation}\label{eq:cond-trunc-allk}
\E\!\left[e^{a_n Y_2}\mathbf 1_{\{Y_2\ge t_n\}}\ \middle|\ d_1=k\right]
=
(1+o(1))\,
\E\!\left[e^{a_n Y_2}\mathbf 1_{\{Y_2\ge t_n\}}\right].
\end{equation}
\end{corollary}

\begin{proof} We will retrace the proof of Lemma \ref{lem:mgf-stab-d2}, but take care to account for the influence of the indicator. For notation, denote
$F(k):=\E\!\left[e^{\lambda(d_2-\mu)}\mathbf 1_{\{d_2\ge r_n\}}\ \middle|\ d_1=k\right]$, $B:=\1\{\{1,2\}\in E(G)\}$, and $X_k:=\#\{\{2,j\}\in E(G): j\ge 3\}.$ As before, with  $N_0:=N-(n-1)$ and $K_0 := n - 2$,
\[
X_k\mid (d_1=k)\ \sim\ \Hyp(N_0,K_0,m-k),
\qquad\text{and}\qquad
X_k \perp\!\!\!\perp B\ \text{ given }d_1=k,
\]
and therefore $d_2=B+X_k$. Thus,
\begin{align}
F(k)
&= e^{-\lambda\mu}\,
\E\!\left[e^{\lambda(B+X_k)}\mathbf 1_{\{B+X_k\ge r_n\}}\ \middle|\ d_1=k\right] \nonumber\\
&= e^{-\lambda\mu}\Big[
\PP(B=0\mid d_1=k)\,\E\!\left[e^{\lambda X_k}\mathbf 1_{\{X_k\ge r_n\}}\right]
+\PP(B=1\mid d_1=k)\,e^\lambda\,\E\!\left[e^{\lambda X_k}\mathbf 1_{\{X_k\ge r_n-1\}}\right]
\Big] \nonumber\\
&= e^{-\lambda\mu}\Big[
\Big(1-\frac{k}{n-1}\Big)\,H(m-k,r_n)
+\frac{k}{n-1}\,e^\lambda\,H(m-k,r_n-1)
\Big], \label{eq:Fk-trunc}
\end{align}
where for $s\in\{0,1,\dots,N_0\}$ and integer $r$ we write
\[
H(s,r):=\E\!\left[e^{\lambda X}\mathbf 1_{\{X\ge r\}}\right],\qquad X\sim \Hyp(N_0,K_0,s).
\]
As in the proof of Lemma~\ref{lem:mgf-stab-d2}, sequential exposure of the draws implies that
changing the sample size by one changes the logarithm of the (untruncated) MGF by at most
$O(|\lambda|/n)$. With truncation, the same one-step computation gains an extra boundary term
proportional to $\PP(X=r-1)$. By Lemma \ref{lemHypTailToGaussianCDF}, uniformly in $1\le s\le m$ and $r\in\{r_n,r_n-1\}$,
\[
\frac{\PP(X=r-1)}{\PP(X\ge r)}=o(1),
\]
so the boundary contribution is $o(1/n)$ relative to $H(s,r)$. Consequently,
\begin{equation}\label{eq:H-step-short}
\sup_{\substack{1\le s\le m\\ r\in\{r_n,r_n-1\}}}
\big|\log H(s,r)-\log H(s-1,r)\big|
\le C\frac{|\lambda|}{n}+o\!\left(\frac{1}{n}\right).
\end{equation}
As the prefactors $\big(1-\frac{k}{n-1}\big)$ and $\frac{k}{n-1}e^\lambda$ are also log-Lipschitz in $k$
with increment $O(|\lambda|/n)$. Repeating the telescoping argument from Lemma~\ref{lem:mgf-stab-d2} yields
\[
\sup_{k,k'}\Big|\log\frac{F(k)}{F(k')} - 1\Big|=o(1),
\]
finishing the proof.
\end{proof}

The next lemma involves the moment generating function of two vertex degrees $d_1$ and $d_2$ in $G(n,m)$. The proof mirrors the structure of Lemma \ref{lem:mgf-stab-d2}, but with modifications to account for the fact that $d_1 $ and $ d_2$ are slightly correlated.

\begin{lemma} 
\label{lem:mgf-stab-d3}
Let $G\sim \PP_0 = G(n,m)$, let $N:=\binom{n}{2}$, and assume $m\le (1-\delta)N$ for some fixed
$\delta\in(0,1)$. Let $\mu:=2m/n$ and let $\lambda=\lambda_n$ satisfy $|\lambda|\to 0$.
Then uniformly for every $k\in\{0,1,\dots,n-1\}$,
\[
\E\!\left[e^{\lambda(d_1+d_2-2\mu)}\mid d_3=k\right]
=\E\!\left[e^{\lambda(d_1+d_2-2\mu)}\right]\cdot (1+o(1)).
\]
\end{lemma}

\begin{proof} Let $B_{13}:=\1\{\{1,3\}\in E(G)\}$, $B_{23}:=\1\{\{2,3\}\in E(G)\}$, and $B_{12}:=\1\{\{1,2\}\in E(G)\}$.
Let us condition on $d_3=k$. Then we have the following observations.
\begin{itemize}
\item The $k$ edges incident to vertex $3$ form a uniformly random $k$-subset of the $(n-1)$ possible edges incident to vertex $3$.
\item The remaining $s:=m-k$ edges form a uniformly random $s$-subset of the $N_1:=N-(n-1)$ edges not incident to $3$.
\item These two random subsets are independent.
\end{itemize}
In particular, $(B_{13},B_{23})$ is independent of the edges not incident to $3$ given $d_3=k$. Now,
\[
d_1+d_2 = (B_{13}+B_{23}) + W_s,
\]
where $W_s$ is defined as
\[
W_s:= 2\cdot \1\{\{1,2\}\in S\} \;+\!\!\!\sum_{j\in [n]\setminus\{1,2,3\}}
\Big(\1\{\{1,j\}\in S\}+\1\{\{2,j\}\in S\}\Big).
\]
In other words, among edges in $S$, every edge incident to exactly one of $\{1,2\}$ contributes weight $1$,
while the special edge $\{1,2\}$ contributes weight $2$. Therefore, for each possible $k$,
\begin{equation}
\label{eq:Gk-def}
\E\!\left[e^{\lambda(d_1+d_2-2\mu)}\mid d_3=k\right]
=e^{-2\lambda\mu}\,\E\!\left[e^{\lambda(B_{13}+B_{23})}\mid d_3=k\right]\,
\widetilde M_1(s),
\end{equation}
where $s=m-k$ and
\[
\widetilde M_1(s):=\E\!\left[e^{\lambda W_s}\right]
\quad\text{for }S\text{ a uniform $s$-subset of the $N_1$ edges not incident to $3$}.
\]
We will now bound $\widetilde M_1(s-1)/\widetilde M_1(s)$. Expose the $s$ edges of $S$ sequentially without replacement from the $N_1$-edge population.
At each step, the newly revealed edge has weight $0$, $1$, or $2$ in $W_s$.
Let $\Delta_s\in\{0,1,2\}$ be the weight contribution of the $s$-th chosen edge, and let $\mathcal F_{s-1}$
be the history up to step $s-1$. Then
\[
\widetilde M_1(s)=\E\!\left[e^{\lambda W_{s-1}}\E\!\left[e^{\lambda\Delta_s}\mid \mathcal F_{s-1}\right]\right].
\]
Given $\mathcal F_{s-1}$, the probability that the next edge has nonzero weight is at most
\[
q_s \le \frac{K_1}{N_1-(s-1)},
\qquad
K_1:=\#\{\text{edges (not incident to 3) that touch 1 or 2}\}=2n-5.
\]
This is because among edges not incident to $3$, the edges touching $\{1,2\}$ are the single edge $\{1,2\}$ plus the $2(n-3)$ edges $\{1,j\},\{2,j\}$ for $j\notin\{1,2,3\}$, hence $K_1=1+2(n-3)=2n-5$. Moreover, the worst-case weight is $2$. Hence
\[
\E\!\left[e^{\lambda\Delta_s}\mid \mathcal F_{s-1}\right]
\in \big[\,1+(e^{\lambda}-1)\cdot 0,\ \ 1+(e^{2\lambda}-1)\,q_s\,\big].
\]
As before, $m\le (1-\delta)N$ implies $N_1-(s-1)\ge (\delta/2)N_1$ for large $n$.
Since $K_1=\Theta(n)$ and $N_1=\Theta(n^2)$, we have $q_s\le C/n$ for a constant $C=C(\delta)$.
Therefore, for all large $n$ and all $1\le s\le m$,
\[
\Big|\log\frac{\widetilde M_1(s)}{\widetilde M_1(s-1)}\Big|
\le \Big|\log\Big(1+(e^{2\lambda}-1)\frac{C}{n}\Big)\Big|
\le \frac{C'}{n}|e^{2\lambda}-1|
\le C''\frac{|\lambda|}{n},
\]
As $|\lambda|\to 0$, we have
\begin{equation}
\label{eq:M1-step}
\sup_{1\le s\le m}\Big|\log \widetilde M_1(s)-\log \widetilde M_1(s-1)\Big|
\le C_4\frac{|\lambda|}{n}.
\end{equation}

For the $(B_{13}+B_{23})$ factor, note that given $d_3=k$, the $k$ incident edges to $3$ form a uniform $k$-subset of $(n-1)$ edges,
so $(B_{13},B_{23})$ is a pair of coordinates of a uniform $k$-subset indicator vector.
In particular,
\[
\E[e^{\lambda(B_{13}+B_{23})}\mid d_3=k]
= 1 + (e^\lambda-1)\,\E[B_{13}+B_{23}\mid d_3=k]
+ (e^\lambda-1)^2\,\E[B_{13}B_{23}\mid d_3=k],
\]
and the exact Hypergeometric formulas give
\[
\E[B_{13}+B_{23}\mid d_3=k]=\frac{2k}{n-1},
\qquad
\E[B_{13}B_{23}\mid d_3=k]=\frac{k(k-1)}{(n-1)(n-2)}.
\]
Hence $\E[e^{\lambda(B_{13}+B_{23})}\mid d_3=k]$ is a smooth function of $k/(n-1)$ with Lipschitz
constant $O(|e^\lambda-1|)=O(|\lambda|)$, so incrementing $k$ by $1$ changes its logarithm by at most $C_5|\lambda|/n$ and thus for all $0\le k\le n-2$,
\begin{equation}
\label{eq:Bstep}
\Big|\log \E[e^{\lambda(B_{13}+B_{23})}\mid d_3=k+1]
-\log \E[e^{\lambda(B_{13}+B_{23})}\mid d_3=k]\Big|
\le C_5\frac{|\lambda|}{n}.
\end{equation}
To put things together, denote $
G(k):=\E\!\left[e^{\lambda(d_1+d_2-2\mu)}\mid d_3=k\right].$ By \eqref{eq:Gk-def}, $\log G(k)= -2\lambda\mu + \log \E[e^{\lambda(B_{13}+B_{23})}\mid d_3=k] + \log \widetilde M_1(m-k)$.
Using \eqref{eq:M1-step} and \eqref{eq:Bstep},
\[
|\log G(k+1)-\log G(k)|
\le C_6\frac{|\lambda|}{n}
\]
uniformly and summing over at most $n$ steps yields $
\sup_{k,k' \in [n-1]}\Big|\log\frac{G(k)}{G(k')}\Big|\le C_6|\lambda|=o(1).$
Thus $\max_k G(k)=(1+o(1))\min_k G(k)$, so the result follows.
\end{proof}

\end{document}